\documentclass[a4paper,10pt,leqno,english]{smfart}
\usepackage{aeguill}
\usepackage{enumerate}
\usepackage{amssymb,amsmath,latexsym,amsthm}
\usepackage[T1]{fontenc}
\usepackage{newpxtext,newpxmath}
\usepackage[scaled=0.92]{helvet}
\usepackage{microtype}
\usepackage{smfthm}      
\usepackage{geometry}   
\usepackage{url}   
\usepackage[french, main=english]{babel}
\usepackage{mathrsfs}
\usepackage{xcolor}
\usepackage{comment}
\definecolor{Duruo}{HTML}{426ab3}
\definecolor{Zaohong}{HTML}{c32136}
\definecolor{Jianghuang}{HTML}{e2c027}
\definecolor{Jiaoqing}{HTML}{87723e}
\usepackage{doi}
\usepackage{hyperref}
\usepackage{relsize}
\usepackage{verbatim}
\usepackage{tikz}

\hypersetup{
    pdftoolbar=true,        % show Acrobat’s toolbar?
    pdfmenubar=true,        % show Acrobat’s menu?
    pdffitwindow=false,     % window fit to page when opened
    pdfstartview={FitH},    % fits the width of the page to the window
    pdftitle={},    % title
    pdfauthor={},     % author
    colorlinks=true,       % false: boxed links; true: colored links
    linkcolor=Duruo,          % color of internal links
    citecolor=Zaohong,        % color of links to bibliography
    urlcolor=Jianghuang,        % color of links to external link
    filecolor=Jiaoqing,      % color of file links
            }           
    
\numberwithin{equation}{section}
\theoremstyle{plain}
\newtheorem{bigthm}{Theorem}
\newtheorem{bigcond}[bigthm]{Condition}
\renewcommand{\thebigthm}{\Alph{bigthm}} % "letter-numbered" theorems

\makeatletter
\newcommand{\settheoremtag}[1]{% \settheoremtag{<tag>}
	\let\oldthebigthm\thebigthm% Store \thetheorem
	\renewcommand{\thebigthm}{#1}% Redefine it to a fixed value
	\g@addto@macro\endbigcond{% At \end{theorem}, ...
		\addtocounter{bigthm}{-1}% ...restore theorem counter value and...
		\global\let\thebigthm\oldthebigthm}% ...restore \thetheorem
}
\makeatother

\renewcommand{\a}{\alpha}
\newcommand{\bC}{\mathbb{C}}

\newcommand{\bQ}{\mathbb{Q}}
\newcommand{\bR}{\mathbb{R}}
\newcommand{\bP}{\mathbb{P}}

\newcommand{\cO}{\mathcal{O}} 
\newcommand{\cA}{\mathcal{A}}

\newcommand{\cC}{\mathcal{C}}

\newcommand{\cE}{\mathcal{E}}

\newcommand{\cI}{\mathcal{I}}

\newcommand{\rH}{\mathrm{H}}

\newcommand{\ba}{\mathbf{a}}
\newcommand{\bb}{\mathbf{b}}
\newcommand{\bc}{\mathbf{c}}

\newcommand{\la}{\langle}
\newcommand{\ra}{\rangle}
\renewcommand{\d}{\partial}
\renewcommand{\b}{\bar}
\newcommand{\vphi}{\varphi}

\renewcommand{\r}{\rho}
\renewcommand{\t}{\tau}

\newcommand{\MA}{\mathrm{MA}}
\newcommand{\Ric}{\mathrm{Ric}}

\newcommand{\vol}{\mathrm{vol}}
\newcommand{\lcs}{\mathrm{Nklt}}
\newcommand{\Capa}{\mathrm{Cap}}

\newcommand{\psh}{\mathrm{PSH}}
\newcommand{\supp}{\mathrm{supp}}

\begin{document}
    
\title{$L^\infty$-estimate of K\"ahler-Einstein potentials on stable varieties}
\date{\today}
\author{Rui Tang}
\email{rui.tang@math.univ-toulouse.fr / trui69954@gmail.com}
\address{Institut de Mathématiques de Toulouse; UMR 5219, Université de Toulouse; CNRS, UPS, 118 route de Narbonne, F-31062 Toulouse, France}

\begin{abstract}
We study the asymptotic behavior of Kähler--Einstein potentials on stable varieties near the singularities. Using iterated logarithmic functions associated with a defining function of the non-klt locus, we obtain refined lower bounds for the Kähler--Einstein potential, improving previous estimates of Di Nezza--Guedj--Guenancia and Datar--Fu--Song. Under additional assumptions on the log resolution, we also establish upper bounds. The proofs are based on the construction of explicit subsolutions and supersolutions for degenerate complex Monge--Ampère equations together with refined integrability estimates in pluripotential theory.
\end{abstract}

\maketitle
\tableofcontents

\section*{Introduction}
The K\"ahler-Einstein metric is a central object in complex differential geometry. Given a complex manifold $X$, a K\"ahler-Einstein metric on $X$ is a K\"ahler metric whose Ricci curvature is proportional to the metric itself. The existence and uniqueness of Kähler–Einstein metrics on canonically polarized (i.e. $K_X$ ample) compact complex manifolds were established in the celebrated works of Aubin~\cite{aubin1976equations} and Yau~\cite{Yau78}. 

Over the past two decades, this theory has been extended to canonically polarized varieties $X$ with mild singularities arising in the Minimal Model Program (MMP). Eyssidieux–Guedj–Zeriahi~\cite{EGZ09} constructed Kähler–Einstein metrics on varieties with Kawamata log terminal (klt) singularities, while Berman–Guenancia~\cite{BG13stable_var} treated the log canonical (lc) case. More generally, if $X$ has semi-log canonical singularities, Berman–Guenancia~\cite{BG13stable_var} showed that $X$ admits a K\"ahler-Einstein metric if and only if $X$ is a stable variety in the sense of Kollár-Shepherd-Barron and Alexeev (KSBA).
These singular K\"ahler-Einstein metrics are defined to be a K\"ahler form $\omega_{KE}=\omega+dd^c\vphi_{KE}$ on the regular part $X_{reg}$ of $X$, where $\omega\in c_1(K_X)$, such that 
$$\Ric(\omega_{KE})=-\omega_{KE} \text{ and } \int_{X_{reg}}\omega_{KE}^n=c_1(K_X)^n.$$
Here $\vphi_{KE}$ is called a K\"ahler-Einstein potential. Due to the complexity of singularities, the behavior of $\omega_{KE}$ near $X_{sing}$ is quite mysterious. A fundamental open problem is the following
\paragraph*{\textbf{Question}} \textit{What is the asymptotic behavior of $\vphi_{KE}$ near $X_{sing}$?}
\paragraph*{\textbf{Known results}}If $X$ has klt singularities, the Kähler–Einstein potential is known to be locally bounded near $X_{sing}$. In contrast, if $X$ has lc singularities, the potential necessarily diverges to $-\infty$ near the non-klt locus $\lcs(X)$ (see Definition \ref{lcs}).
Song~\cite{Song} proved that the Kähler–Einstein potential is locally bounded outside the non-klt locus. More recently, Di Nezza–Guedj–Guenancia~\cite{DGG23familiesKE,DGG23familiesKEcorr} obtained a more precise estimate: 

\paragraph*{\textbf{Fact A}} For any $\epsilon>0$, there exists a constant $C_\epsilon$ such that
\begin{equation}\label{DGG lowerbdd}
    \vphi_{KE}\geq-(n+\nu_X+\epsilon)\log(-\log|s|)+C_\epsilon,
\end{equation}
where $n=\dim X$, $\nu_X\in\{1,\dots,n\}$ is an integer defined as in Definition \ref{define_nu_X} below, and $(s=0)$ is any reduced divisor containing the non-klt locus of $X$.

\medskip

On the other hand, Datar–Fu–Song~\cite[Theorem 1.1.(2)]{DFS23} proved the existence of local K\"ahler-Einstein metric near isolated log canonical singularities and obtained a lower bound with a worse coefficient:

\paragraph*{\textbf{Fact B}} Let $(X,x)$ be a germ of an isolated log canonical (non-log terminal) singularity. let $\pi:Y\to X$ be a log resolution, and let $E$ be the reduced exceptional divisor. Fix a defining section $s_E$ of $E$ and a Hermitian metric $|\cdot|_h$ on $\cO_Y(E)$. Then there exists a local K\"ahler-Einstein metric $\omega=dd^c\vphi_{KE}$ near $x$, and for any $\epsilon>0$, there exists a constant $C_\epsilon$ such that 
\begin{equation}\label{DFS lowerbdd}
    \pi^*\vphi_{KE}\geq-(2n+\epsilon)\log(-\log |s_E|_h^2)+C_\epsilon.
\end{equation}

\medskip

Nevertheless, a quantitative description of growth of the potential near the non-klt locus is still missing in general, except for some particular examples.  

\begin{exem}
    Let $X$ be Baily-Borel-Satake compactification of a ball quotient of dimension $n$, then it has isolated singularities and admits a log resolution $\pi:Y\to X$ whose exceptional divisor over each singularity is an abelian variety with discrepancy $-1$. In this case, $\nu_X=1$ and 
    $$\pi^*\vphi_{KE}=-(n+1)\log(-\log |s_E|^2)+\cO(1),$$
    where $s_E$ is a defining section of the exceptional divisor.
\end{exem}

\begin{exem}(see \cite[p. 54-57]{kobayashi85})
    Let $X$ be a Hilbert modular surface, then X has isolated singularities. Let $\pi:Y\to X$ be the minimal resolution, then the exception divisor over each singularity is a cycle of $\bP^1$'s, and each $\bP^1$ has discrepancy $-1$. In this case, $n=\nu_X=2$, and $$\pi^*\vphi_{KE}=-4\log(-\log |s_E|^2)+\cO(1),$$
    where $s_E$ is a defining section of the exceptional divisor.
\end{exem}

\medskip

The goal of this paper is to improve the lower bound (\ref{DGG lowerbdd}), (\ref{DFS lowerbdd}) by using iterated logarithmic functions associated with the non-klt locus, and to obtain an explicit upper bound which also diverges to $-\infty$ near $\lcs(X)$ for the K\"ahler–Einstein potential in some special cases.

We first introduce a function $\r$ on a general stable variety which cut out the non-klt locus:
\paragraph*{\textbf{Set-up}}\label{setup}
(\textit{Cut out non-klt locus}) Let $X$ be a stable variety of dimension $n$. Let $L$ be a (very) ample line bundle on $X$, and let $|\cdot|_h$ be a Hermitian metric on $L$. Let $\cI$ be the ideal sheaf of the non-klt locus of $X$ (see Definition \ref{lcs}). After replacing $L$ by a large multiple of it, we can find some sections $\sigma_1,\dots,\sigma_k\in \rH^0(X,L)$ such that they globally generate $\cI \otimes L$. Then the vanishing set of the function $\r:= \sum_{i=1}^k|\sigma_i|_h^2$ is precisely the non-klt locus of $X$. By rescaling the Hermitian metric, we always assume $\r<e^{-1}$, then $-\log\r>1$.

\medskip

The two examples above show that the asymptotic behavior of $\vphi_{KE}$ near $\lcs(X)$ is related to the combinatorial structure of exceptional divisors on a log resolution. Motivated by this, we introduce the following
\begin{defi}\label{define_nu_X}
    Let $X$ be a compact complex manifold of dimension $n$, and let $D=\sum_{i\in I}D_i$ be a reduced divisor with simple normal crossing support. We define the following function $m_{(X,D)}:D\to\{1,2,\dots,n\}$ by 
    $$m_{(X,D)}(x)=\#\{i\in I : x\in \supp(D_i)\}.$$
    And we define $\nu(X,D)$ to be the maximal number of $D_i$'s that intersect simultaneously:
    \begin{align*}
        \nu(X,D)&=\max_{x\in D}m_{(X,D)}(x)
    \end{align*}
    Then we always have $1\leq\nu(X,D)\leq n$.
    
    Now let $X$ be a stable variety, let $\pi:(Y,E=\sum_i a_iE_i)\to X$ be a log resolution of $X$, where $E_i$ is either an exceptional divisor or set above the codimension 1 singularities. Then we have $a_i\geq-1$. Let $E_{lc}=\sum_{a_i=-1}E_i$. We denote $\nu_X(\pi)=\nu(Y,E_{lc})$ and introduce the following invariants:
    $$\nu_X=\inf_{\pi}\nu_X(\pi),$$
    where $\pi$ runs over all log resolutions of $X$. Since $\nu_X(\pi)\in \{1,\dots,n\}$ for every log resolution $\pi$, the infimum is always attained by some log resolution.
\end{defi}

\begin{rema}
    The quantity $\nu(X,D)-1$
    coincides with the dimension of the dual complex associated to $(X,D)$ in the sense of de Fernex, Koll\'ar and Xu, see \cite{dualcomplex}.
\end{rema}

Our results can be divided into two parts: the first part is an unconditional lower bounded, which improves the lower bound obtained in (\ref{DGG lowerbdd}); the second part concerns an upper bounded (which also grows in $\log\log$ form) under some geometrical assumption.

\subsection{Unconditional lower bound}
\begin{bigthm}[= \textit{Theorem \ref{lowerbdd=thma}}]\label{thma}
    Let $X$ be a stable variety of dimension $n$, and $\omega_X \in c_1(K_X)$ be a smooth K\"ahler form. 
    Let $\omega=\omega_X+dd^c\varphi_{KE}$ be the K\"ahler-Einstein metric on $X$.
    Then for any $k\geq 1$ and $\epsilon>0$, there exists a constant $C_{k,\epsilon}$ such that 
        $$\varphi_{KE} \geq -(n+\nu_X)\sum_{j=1}^k\log^{(j)}(-\log\r) -\epsilon\log^{(k)}(-\log\r) + C_{k,\epsilon}.$$
    Here $\r$ is the function defined in \hyperlink{setup}{Set-up}, and $\log^{(j)}(\cdot):=\underbrace{\log\circ\dots\circ\log}_{\text{j times}}\ (\cdot)$.
\end{bigthm}
\begin{rema}
    When $k=1$, our result recovers the lower bound (\ref{DGG lowerbdd}):
    $$\vphi_{KE}\geq-(n+\nu_X+\epsilon)\log(-\log\r)+C_\epsilon.$$
    The above two examples show that this lower bound is almost sharp (at least for isolated singularities), with an error term $-\epsilon\log(-\log\r)$. Our result for $k\geq2$ shows that this error term can be replaced by less singular terms consisting of iterated logarithmic functions.
\end{rema}

As we mentioned before, we know qualitatively that $\vphi_{KE}\to-\infty$ near $\lcs(X)$, but to understand the asymptotic behavior of $\vphi_{KE}$, we would like to have an upper bound which also diverges to $-\infty$ near $\lcs(X)$. In the spirit of the above two examples, some multiple of $-\log(-\log\r)$ should be a candidate. The general case however seems to be difficult, we only prove an upper bound under some additional geometric assumption on the resolution.

\subsection{Conditional upper bound}
Our first upper bound is the following
\begin{bigthm}[= \textit{Theorem \ref{thmb=upperbdd}}]\label{thmb}
    Let $X$ be as in Theorem \ref{thma}. Let $d=\dim X_{sing}$.
    Assume $X$ admits a log resolution $\pi:Y\to X$ satisfying
    \begin{itemize}
        \item[(1)] $\pi$ is a log crepant resolution (i.e. every exceptional divisor of $\pi$ has discrepancy -1);
        \item[(2)] $\pi$ is moreover a log resolution of $(X,\cI)$, where $\cI$ is the ideal sheaf of $\lcs(X)$ ($=X_{sing}$ under condition (1)).
    \end{itemize}
    Then there is a constant $C>0$ such that
    $$\vphi_{KE}\leq -(n-d+1)\log(-\log\r)+C.$$ 
\end{bigthm}
We shall note that this upper bound is far from being sharp, which can already be seen in the case of the Hilbert modular surface. However, the coefficient $(n-d+1)$ can not be improved for a general stable variety. Indeed, if $X$ has codimension 1 singularities, then $d=n-1$ and $n-d+1=2$. Let $\nu:X^\nu\to X$ be the normalization, then there exists a reduced divisor $C$ on $X^\nu$, called the conductor divisor, such that $\nu^*K_X=K_{X^\nu}+C$ and $K_{X^\nu}+C$ is ample. The pull back of the K\"ahler-Einstein metric on $X$ coincide with the K\"ahler-Einstein metric on the pair $(X^\nu,C)$. In \cite{GW16KEboundary_lc}, it is proved that $\omega_{KE}$ has cusp singularities near $C\cap X^\nu_{reg}$, i.e. $\vphi_{KE}=-2\log(-\log|s_C|)+\cO(1)$, where $s_C$ is a section of $\cO_{X^\nu}(C)$ that cuts out $C^\nu$. Thus the coefficient of $-\log(-\log\r)$ can not be larger than $2$.

When $\pi$ has another special form, we will be able to deduce a similar upper bound and a possibly better lower bound:
\begin{bigthm}[= \textit{Theorem \ref{2sideest=thmc}}]\label{thmc}
    Let $X$ be as in Theorem \ref{thma}. Let $d=\dim X_{sing}$. Assume $X$ admits a log resolution $\pi:Y\to X$ that satisfies the following two conditions:
    \begin{itemize}
        \item[(1)] $\pi$ is a log crepant resolution;
        \item[(2')] $\pi$ is a composition of finitely many blow-ups along smooth centers.
    \end{itemize}
    We write $K_Y=\pi^*K_X-E$. Let $s$ be a section of $\cO_Y(E)$ that cuts out $E$ and let $h$ be a Hermitian metric on $\cO_Y(E)$. Then we have
    $$-(n-d+1)\log(-\log|s|^2_{h})+\cO(1)\geq\pi^*\vphi_{KE}\geq-(n+\nu_X(\pi))\log(-\log|s|^2_{h})+\cO(1).$$
\end{bigthm}

\begin{rema}
    By definition we have $\nu_X(\pi)\geq \nu_X$, it is not clear that if the equality holds or not under these assumptions. If the equality holds, then we do get a sharp lower bound that improve the lower bound in Theorem \ref{thma}.
\end{rema}

\begin{rema}
    As we will see, the main difficulty for deducing an upper bound is that, there may exist an exceptional divisor lying over $\lcs(X)$ but with discrepancy $>-1$. It is hard to distinguish this kind of exceptional divisors with the exceptional divisors whose discrepancy $=-1$ from a metric viewpoint.
    
    The condition $(1)$ in the above two theorems rules out this case. In dimension $2$, this condition is always satisfied by taking the minimal resolution. For higher dimensional case, some properties of varieties satisfying this condition is researched in \cite{benoit}.
\end{rema}

\subsection*{Strategy of proof}
We briefly explain the main ideas of the proof. Let
$$\pi:Y\to X$$
be a log resolution. Then constructing the singular K\"ahler-Einstein metric on $X$ boils down to solving the following degenerate Monge-Amp\`ere equation on $Y$:
\begin{equation}\label{MA0}
(\pi^*\omega_X+dd^c\psi)^n
=
e^{\psi}
\prod_i |s_i|_{h_i}^{2a_i}\,dV_Y,
\end{equation}
where $K_Y=\pi^*K_X+\sum_i a_iE_i,$ $s_i$ is a section cutting out $E_i$, and $|\cdot|_{h_i}$ is a Hermitian metric on $\cO_Y(E_i)$. Once we get a solution $\psi$ on $Y$, the $\pi^*\omega_X$-pshness of $\psi$ will ensure that it descends to an $\omega_X$-psh function $\vphi_{KE}$. 

The proof of the lower bound estimate in Theorem \ref{thma} is based on constructing explicit
sub-solutions. The idea comes from \cite{DGG23familiesKE}. More
precisely, under the notation of \hyperref[setup]{Set-up}, we consider functions of the form
$$
\psi_{k,\epsilon}
=
-(n+\nu_X)\sum_{j=1}^k \log^{(j)}(-\log\rho)
-\epsilon \log^{(k)}(-\log\rho).
$$
The proof is divided into three main steps.

\begin{itemize}
    \item First, one proves that $\pi^*\psi_{k,\epsilon}$
are \(\pi^*\omega_X\)-plurisubharmonic after rescaling the 
Hermitian metric $|\cdot|_h$ on $L$ if necessary. Since it is pulled back from $X$, on which $\omega_X$ is a K\"ahler form, this is not difficult (see Lemma \ref{finiteE of psi_k}).
    \item Secondly, one checks that $\pi^*\psi_{k,\epsilon}$ has finite Monge--Ampère energy. This allows us to work inside the finite-energy class, where the \hyperref[comparison]{comparison principle} is available. 
    \item Thirdly, a refined integrability estimate shows that the density obtained from the candidate
    sub-solution satisfies Kołodziej's condition (see Section \ref{density}). Solving an auxiliary Monge--Ampère equation then produces a bounded correction term $u$, such that 
    $\pi^*\psi_{k,\varepsilon}+u$
    is a genuine sub-solution of the Kähler--Einstein equation (\ref{MA0}).
The comparison principle then gives
\[
\pi^*\varphi_{\mathrm{KE}}\geq \psi_{k,\varepsilon}+u,
\]
which proves Theorem \ref{thma}.
\end{itemize}

For Theorem \ref{thmb}, under the
additional assumption that all exceptional divisors have
discrepancy $-1$, we construct an explicit super-solution of the form
$$
-(n-d+1)\log(-\log\rho)+C,
$$
where \(d=\dim X_{\mathrm{sing}}\). The key point is to estimate the
Monge--Ampère measure of $-\log(-\log\rho)$ on the log resolution.  Since $\r$ cuts out the non-klt locus, this function is singular along $E_i$'s that lies above $\lcs(X)$. The assumption $(2)$ in Theorem \ref{thmb} assures that one can write locally
$$\log \pi^*\rho = \sum_{i=1}^p b_i\log|z_i|^2+O(1),$$
where $\cup_{i=1}^p(z_i=0)$ lies above $\lcs(X)=X_{sing}$. The form \(dd^c\log(-\log\pi^*\rho)\) has two types of terms: one involving the curvature term
\(dd^c\log\pi^*\rho\), and one involving
$$\frac{d\log\pi^*\rho\wedge d^c\log\pi^*\rho}
{(-\log\pi^*\rho)^2}.$$
The latter has rank one, which simplifies the expansion of the Monge-Amp\`ere measure, and it has poles of the form $\frac{1}{|z_i|^2}$ near generic points of $(z_i=0),\ 1\leq i\leq p$. Our assumption $(1)$ in Theorem \ref{thmb} assures that this pole can be controlled by the poles in the right hand side of the  K\"ahler-Einstein equation (\ref{MA0}).  The remaining estimates reduce to showing that certain mixed terms
involving \(\pi^*\omega_X\) vanish to sufficient order along the exceptional
divisors. This is where the dimension \(d=\dim X_{\mathrm{sing}}\) enters.

This proves that the above logarithmic function is a super-solution, and
the comparison principle yields the desired upper bound. 

Finally, when the
resolution is obtained by a sequence of blow-ups along smooth centers, we will be above to construct suitable $\pi^*\omega_X$-psh functions directly on $Y$. By taking the combinatorial structure of exceptional divisors into consideration, we will be able to deduce the two-sided estimate of Theorem \ref{thmc}.

\subsection*{Acknowledgements}

The author would like to thank his PhD advisor, Henri Guenancia, for many useful discussions and valuable suggestions on earlier drafts, and for his patient guidance and encouragement. The author is also grateful to his co-advisor, Yuxin Ge, for his constant support and encouragement. Finally, the author would like to thank Shengxuan Zhou and Junyu Meng for many helpful discussions.

\section{Preliminaries}
    In this section, we collect some essential materials for understanding the singular K\"ahler-Einstein metrics. $X$ will denote a compact K\"ahler manifold of dimension $n$ unless otherwise specified. Recall that for a smooth closed real $(1,1)$-form $\theta$, a $\theta$-psh function is a quasi-psh function $\vphi$ (i.e. locally the sum of a smooth function and a psh function) such that $\theta+dd^c\vphi\geq0$ in the sense of current. The set of $\theta$-psh functions is denoted by $\psh(X,\theta)$.

\subsection{Non-pluripolar product.}
    In \cite{BEGZ_2010}, the authors define a non-pluripolar product which sends a $p$-tuple of any closed positive $(1,1)$-currents $(T_1,\dots,T_p)$ to a closed positive $(p,p)$-current $\la T_1\wedge\dots\wedge T_n \ra$ on any complex manifold (not necessarily compact K\"ahler). They proved that this product is always well defined on a compact K\"ahler manifold (cf.\cite[Prop.1.6]{BEGZ_2010}). In particular, the non-pluripolar product $T\to\la T^n\ra$ gives a well-defined measure which puts no mass on a pluripolar set. 
    
    Given a smooth $(1,1)$ form $\theta$ and a $\theta$-psh function $\vphi$, we define the non-pluripolar Monge-Amp\`ere measure of $\vphi$ as $\MA(\vphi):=\la(\theta+dd^c\vphi)^n\ra$. We shall mention that, if $\vphi$ is locally bounded outside a (complete) pluripolar subset $A\subset X$, then $\la(\theta+dd^c\vphi)^n\ra$ is well-defined if and only if the Bedford–Taylor product $(\theta+dd^c\vphi)^n$, which is deﬁned on the open subset $X\setminus A$, has locally ﬁnite mass near each point of A. In this case, $\MA(\vphi)=\la(\theta+dd^c\vphi)^n\ra$ is just the trivial extension of $(\theta+dd^c\vphi)^n$ to $X$ (see \cite[p. 204]{BEGZ_2010}).

\subsection{Big cohomology class and singularity type of currents.}
    A \textit{K\"ahler current} is a positive closed $(1,1)$-current that dominates a smooth K\"ahler form. We say that a cohomology class $\alpha\in H^{1,1}(X,\bR)$ is \textit{big} if it can be represented by a K\"ahler current; it is \textit{pseudo-effective} (psef for short) if it can be represented by a positive current.

    Fix a smooth form $\theta\in\alpha$, if $T_1,T_2$ are two closed positive currents in $\alpha$, we can write $T_i=\theta+dd^c\vphi_i$. We say that $T_1$ is less singular than $T_2$ if their global potentials satisfy $\vphi_1 \geq \vphi_2 + O(1)$. This definition is clearly independent of the choice of $\theta$ and the potentials, hence is well defined. A closed positive current $T_{min}$ in $\alpha$ is said to have \textit{minimal singularities} if it is less singular than any other positive current in $\alpha$. A $\theta$-psh function $\vphi$ is said to have minimal singularities if $\theta+dd^c\vphi$ is so. Note that this current is not unique in general. One way to construct such a current is to define the upper envelop:
    $$V_\theta:=\sup\{\vphi\in\psh(X,\theta) : \vphi\leq0\ \text{on}\ X\}.$$
    This is a well defined $\theta$-psh function once $\a$ is psef. It is clear that $\theta+dd^cV_\theta$ has minimal singularities.

    A positive current $T=\theta+dd^c\vphi$ is said to have \textit{analytic singularities} if locally on $X$ we have 
    $$\vphi=\frac{c}{2}\log\sum_{i=1}^N|f_i|^2+u,$$
    where $u$ is a smooth function and $f_i$'s are holomorphic functions.

\subsection{Currents of full Monge-Amp\`ere mass.}
    Given a smooth form $\theta$ in a big cohomology class $\alpha$, the volume $\vol(\alpha)$ of $\alpha$, introduced in \cite{boucksom2002volume}, satisfies following inequality \cite[Prop.1.20]{BEGZ_2010}:
    $$\int_X\MA(\vphi) \leq \vol(\a),$$
    where $\vphi$ is any $\theta$-psh function. The functions such that the equality holds are said to have \textit{full Monge-Amp\`ere mass}. These functions are also defined as the \textit{finite energy class}, denoted by $\cE(X,\theta)$. This class appears naturally in pluripotential theory, it can also be characterized by weighted Monge-Amp\`ere energy functionals or be viewed as the maximal domain on which $\MA(\vphi)$ can be defined as a measure which does not charge any pluripolar sets, for more details refer to \cite[\S 2]{BBGZ_2013}, \cite[\S 2]{BEGZ_2010} or \cite{GZweightedMA}. What is important to us is that the following comparison principle holds \cite[Coro.2.3]{BEGZ_2010}:
    \begin{prop}[Comparison Principle]\label{comparison}
        For $\vphi,\psi\in\cE(X,\theta)$, we  have
        $$\int_{\{\vphi<\psi\}} \MA(\psi)\leq \int_{\{\vphi<\psi\}} \MA(\vphi).$$
    \end{prop}

    Now, for simplicity we assume $\theta$ is semi-positive and big. In this case, $V_\theta\equiv0$, and a $\theta$-psh function has minimal singularities if and only if it is bounded. To check whether a $\theta$-psh function has finite energy is generally not an easy thing. However, for an important subset $\cE^1(X,\theta)\subset \cE(X,\theta)$ we do have a useful criterion. We now recall the definition of this subset. First, if $\vphi\in\psh(X,\theta)\cap L^\infty(X)$, we set
    $$E(\vphi):= \frac{1}{(n+1)\vol(\a)}\sum_{j=0}^n\int_X\vphi\ \la(\theta+dd^c\vphi)^j \wedge \theta^{n-j}\ra,$$
    and for any $\theta$-psh function $\vphi$, we define
    $$E(\vphi):=\inf\left\{ E(\psi) : \psi\in\psh(X,\theta)\cap L^\infty(X),\ \psi\geq\vphi \right\}\in[-\infty,+\infty).$$
    Then we set
    $$\cE^1(X,\theta)=\left\{ \vphi\in\psh(X,\theta) : E(\vphi)>-\infty \right\}.$$
    \cite[Proposition 2.11]{BEGZ_2010} shows that: for a $\theta$-psh function $\vphi\in\psh(X,\theta)$, 
    $$\vphi\in\cE^1(X,\theta)\Longleftrightarrow\vphi\in\cE(X,\theta)\text{ and }\int_X\vphi\MA(\vphi)>-\infty,$$
    Then it is clear that $\cE^1(X,\theta)\subset\cE(X,\theta)$, as we mentioned before.
    
    Now given a Borel subset $K\subset X$, the \textit{capacity} of $K$ is defined to be
    $$\Capa_\theta(K)=\sup \left\{\int_K\MA(\vphi) :  \vphi\in\mathrm{PSH}(X,\theta),\ 0\leq\vphi\leq1\right\}.$$
    To check that if $\vphi$ is in $\cE^1$ or not, it suffices to compute the capacity decay of sublevel sets \cite[Lemma 2.9]{BBGZ_2013}:
    \begin{lemm}\label{cap decay}
        Let $\vphi\in\mathrm{PSH}(X,\theta)$. If
        $$\int_{t=0}^{+\infty}t^n\Capa_\theta(\vphi<-t)<+\infty,$$
        then $\vphi\in\cE^1(X,\theta)$.
    \end{lemm}

\subsection{Condition K}
    Due to the breakthrough work of Kołodziej \cite{kolodziej1998complex}, we introduce the following notion (see also \cite[Section 1.3]{henriDiameter}):
    \begin{defi}[Condition (K)]\label{Condition_K}
        We say that a function $w:[0,\infty)\to[0,\infty)$ satisfies \hyperref[Condition_K]{Condition (K)} if it is convex increasing, and there is an increasing function $h:[0,\infty)\to[0,\infty)$ satisfying
        $$\int^\infty\frac{1}{h(t)}dt<\infty,$$
        such that $w(t)=t \cdot (\log(1+t))^n \cdot h^n\circ(\log(1+\log(1+t)))$.
    \end{defi}
        \begin{itemize}
            \item if $h(t)=(e^{e^t-1}-1)^{\epsilon/n}\cdot(e^t-1)^{-1}$, then $w(t)=t^{1+\epsilon};$
            \item if $h(t)=(e^{t}-1)^{\epsilon/n}$, then $w(t)=t\cdot(\log(1+t))^{n+\epsilon}$;
            \item if $h(t)=t^{1+\epsilon/n}$, then $w(t)=t\cdot(\log(1+t))^n\cdot(\log(1+\log (1+t)))^{n+\epsilon}$.
        \end{itemize}
    
    The following theorem is proved in \cite[Theorem 2.5.2]{kolodziej1998complex} (together with theorems for semi-positive and big form developed later in \cite{GZ05cap}, \cite{GZweightedMA},\cite{GZdegenerateCMA}).
    
    \begin{theo}
        Let $X$ be a compact K\"ahler manifold of dimension $n$. Let $\theta$ be a semi-positive and big $(1,1)$-form, and let $dV_X$ be a a smooth volume form on $X$. Suppose that $w$ is a function satisfying \hyperref[Condition_K]{Condition (K)}.
        Then for any non-negative $f\in L^1(X)$ such that
        \begin{equation}\label{Condition on density}
            \int_Xw\circ f\ dV_X<\infty,
        \end{equation}
        the Monge-Amp\`ere equation $(\theta+dd^c\vphi)^n=fdV_X$ has a unique bounded solution $\vphi\in \psh(X,\theta)\cap L^\infty(X)$ normalised by $\sup_X\vphi=0$.
    \end{theo}

    As a corollary, we have
    \begin{coro}\label{solve KEMA}
        Under the same condition as above, the Monge-Am\`ere equation
        $(\theta+dd^c\vphi)^n=e^\vphi fdV_X$ has a unique bounded solution $\vphi\in \psh(X,\theta)\cap L^\infty(X)$.
    \end{coro}

\subsection{Log canonical pairs and stable varieties.} 
We will use some terminology coming from the minimal model program (cf.\cite{kollar-mori}). 
By definition, a \textit{pair} $(X,D)$ consists of a complex projective normal variety $X$ and an effective Weil $\bQ$-divisor $D$.
Assume $K_X+D$ is $\bQ$-Cartier and take a log resolution $\pi:X'\to X$, then there are rational numbers $a_i$ such that 
$$K_{X'}=\pi^*(K_X+D)+\sum a_iE_i,$$
here $E_i$'s are either exceptional divisors or components of the strict transform of $D$. The rational number $a_i$ is called the \textit{discrepancy} of $E_i$.

We say that the pair $(X,D)$ is a \textit{log canonical} (\textit{lc} for short) pair if $K_X+D$ is $\bQ$-Cartier and $a_i\geq -1,\ \forall i$; it is a \textit{Kawamata log terminal} (\textit{klt} for short) pair if $K_X+D$ is $\bQ$-Cartier and $a_i> -1,\ \forall i$. These definitions are independent of the choice of $\pi$.

Stable varieties appear naturally when one considers compactifications of the moduli space of canonically polarized smooth varieties. They have semi-log canonical singularities:
    \begin{defi}
        Let $X$ be a projective projective variety, we say that $X$ has \textit{semi-log canonical} (\textit{slc} for short) singularity if
        \begin{itemize}
            \item[1.] $X$ is Gorenstein in codimension $1$ and satisfies Serre's $S_2$ condition;
            \item[2.] $K_X$ is $\bQ$-Cartier;
            \item[3.] Let $\nu :X^\nu\to X$ be the normalization of $X$ and $C_{X^\nu}$ be the conductor divisor on $X^{\nu}$. Then the pair $(X^\nu,C_{X^\nu})$ is a log canonical pair.
        \end{itemize}
    \end{defi}

    If $X$ is slc, then $C_{X^\nu}$ is necessarily reduced, and we have $\nu^*K_X = K_{X^\nu}+C_{X^\nu}$. Intuitively, the codimension $1$ singularities of $X$ are ordinary nodes. For more details about this notion we refer to \cite{kovacs-singstable}.

    \begin{defi}
        We say that $X$ is a \textit{stable variety} if it has slc singularity and the canonical divisor $K_X$ is $\bQ$-ample. 
    \end{defi}

    \begin{defi}\label{lcs}
        Let $X$ be a stable variety, let $\nu:X^\nu\to X$ be its normalization. Take a log resolution $\mu:Y\to X^\nu$ of the lc pair $(X^\nu,C_{X^\nu})$ and let $f=\nu\circ\mu$. Then we have
        $$K_Y = f^*K_X + \sum a_iE_i,$$
        where $a_i\geq-1$. The \textit{non-klt locus} of $X$ is defined as the Zariski closed subset:
        $$\lcs(X)=\bigcup_{\{a_i=-1\}}f(E_i).$$
    \end{defi}
    This definition does not depend on the choice of log resolution $\mu$, and $\lcs(X)$ coincide with the complement of the set of points where $X$ has klt singularities.

\subsection{Singular K\"ahler-Einstein metrics.}

    There are several equivalent approaches to define K\"ahler-Einstein metric on a stable variety, the simplest one is perhaps the following (cf, \cite[Thm.2.10]{BG13stable_var}):
    
    \begin{defi}
        Let $X$ be a stable variety of dimension $n$. A \textit{K\"ahler-Einstein metric} $\omega$ on $X$ is a smooth K\"ahler form on $X_{reg}$ such that:
        \begin{itemize}
            \item $\Ric(\omega)=-\omega$ on $X_{reg}$,
            \item $\int_{X_{reg}}\omega^n=c_1(K_X)^n$.
        \end{itemize}
    \end{defi}

    This metric extends globally to define a current in $c_1(K_X)$. If we take a log resolution $f:Y\to X$, and we write $K_Y = f^*K_X + \sum a_iE_i$, then constructing the K\"aher-Einstein metric boils down to solving the following Monge-Amp\`ere equation on $Y$:
    \begin{equation*}
            (f^*\omega_X + dd^c \varphi)^n=e^{\varphi} \cdot \prod|s_i|^{2a_i} \cdot dV_Y,
    \end{equation*}
    where $\omega_X\in c_1(K_X)$ is a smooth K\"ahler form on $X$, $dV_Y$ is a smooth volume form on $Y$, $s_i$ is a section of $\cO_Y(E_i)$ such that $E_i=(s_i=0)$, and $|\cdot|_{h_i}$ is a Hermitian metric on $\cO_Y(E_i)$.

\section{Integrability of density.}\label{density}

Motivated by the logarithmic structures appearing in Kołodziej’s theory, we introduce the following functions $\chi_k:[0,+\infty)\to[1,+\infty)$ by
\begin{equation}\label{define chi_k}
    \begin{aligned}
    &\chi_k(t)=\left\{\begin{aligned}
        &1+t,\quad k=0\\
        &1+\log(\chi_{k-1}),\quad k>0
    \end{aligned}\right.\\
    &\psi_k=-\chi_k(-\log\r).
\end{aligned}
\end{equation}

\medskip

We note that 
$\chi_k(t) \sim_{t\to\infty} \underbrace{\log\circ\dots\circ\log}_{\text{k times}}\ (t),$
and thus
$$\psi_k\sim -\underbrace{\log\circ\dots\circ\log}_{\text{k times}}\ (-\log\r)$$
near the non-klt locus.

Now let $X$ be a stable variety, we fix a log resolution $\pi:Y\to X$ such that $\nu_X(\pi)=\nu_X$. For simplicity, we will omit the subscript $X$ and write $\nu$ for $\nu_X$. Let $s_i$ be a section of $\cO_Y(E_i)$ such that $E_i=(s_i=0)$. Let $|\cdot|_{h_i}$ be a Hermitian metric on $\cO_Y(E_i)$. Given a positive integer $k\geq 1$ and a small real number $\epsilon>0$, consider the following density function on $Y$:
$$g_{k,\epsilon}:=\exp\left( (n+\nu)\sum_{j=1}^{k}\pi^*\psi_j+\epsilon\pi^*\psi_{k} \right)\cdot \prod_i|s_i|_{h_i}^{2a_i}.$$
we show that $g_{k,\epsilon}$ is integrable with respect to a suitable weight function $w_k$:

\begin{prop}\label{integrable of g_km}
    Let $dV_Y$ be a smooth volume form on $Y$. Let $\eta>0$ be a small number such that $n\eta<\epsilon$. Let 
    \begin{equation*}
        \begin{aligned}
            &h_k(t):=\left\{\begin{aligned}
                &\chi_0\cdot\chi_1\cdots\chi_{k-2}\cdot\chi_{k-2}^{\eta}(t),&k\geq2\\
                &(e^t-1)^\eta,&k=1
            \end{aligned} \right.\\
            &w_k(t):=t\cdot(\log(1+t))^n\cdot h_k^n\circ(\log(1+\log(1+t)))\ .
        \end{aligned}
    \end{equation*}
    Then $w_k$ satisfies \hyperref[Condition_K]{Condition (K)}  (for large $t$), and we have
    $$\int_Yw_k(g_{k,\epsilon})\ dV_Y<\infty.$$
\end{prop}
\begin{proof}
    We only proof the case $k\geq2$, the proof for $k=1$ is completely the same.
    
    To simplify the notation, we omit $\pi^*$ and view $\r,\ \psi_k$ as functions defined on $Y$. And for two functions $f,g$ on $Y-\cup_i\supp(E_i)$, we say that $f\lesssim g$ if there exists a constant $C>0$ such that $f\leq C(g+1)$, we say that $f\sim g$ if $f\lesssim g$ and $g\lesssim f$.
    
    We define:
    $$|s_{klt}|^2=\prod_{0>a_i>-1} |s_i|_{h_i}^{-2a_i},\quad |s_{lc}|^2=\prod_{a_i=-1}|s_i|_{h_i}^2.$$
    Without loss of generality, we may assume $|s_{klt}|^2,|s_{lc}|^2<e^{-1}$, and $\prod_{a_i>0}|s_i|_{h_i}^2<1$. Then we have
    \begin{align*}
        g_{k,\epsilon}&=\left( \prod_{j=1}^k\frac{e^{-1}}{-\psi_{j-1}} \right)^{n+\nu} \cdot \left( \frac{e^{-1}}{-\psi_{k-1}} \right)^\epsilon\cdot \frac{\prod_{a_i>0}|s_i|_{h_i}^2}{|s_{klt}|^2|s_{lc}|^2}\\
        &\lesssim\left( \prod_{j=1}^k\frac{1}{\chi_{j-1}(-\log\r)} \right)^{n+\nu} \cdot \left( \frac{1}{\chi_{k-1}(-\log\r)} \right)^\epsilon\cdot \frac{1}{|s_{klt}|^2|s_{lc}|^2}.
    \end{align*}
    Recall that $\r=\sum|\sigma_i|_h^2$.  Since each $\sigma_i$ vanishes on $\lcs(X)$, it vanishes at least to order 1 along the divisors $E_i$ with $a_i=-1$. Then we have $\r\leq C |s_{lc}|^2 $; thus,
    \begin{equation}\label{estimate_g}
        g_{k,\epsilon}\lesssim \left( \prod_{j=1}^k\frac{1}{\chi_{j-1}(-\log|s_{lc}|^2)} \right)^{n+\nu} \cdot \left( \frac{1}{\chi_{k-1}(-\log|s_{lc}|^2)} \right)^\epsilon\cdot \frac{1}{|s_{klt}|^2|s_{lc}|^2}.
    \end{equation}
    Similarly, we have the estimate for $\log(1+g_{k,m})$:
    \begin{equation}\label{estimate_logg}
    \begin{aligned}
        \log(1+g_{k,\epsilon})&\lesssim\log\left( 1+\left( \prod_{j=1}^k\frac{1}{\chi_{j-1}(-\log|s_{lc}|^2)} \right)^{n+\nu} \cdot \left( \frac{1}{\chi_{k-1}(-\log|s_{lc}|^2)} \right)^\epsilon\cdot \frac{1}{|s_{klt}|^2|s_{lc}|^2} \right)  \\
        &\lesssim(n+\nu)\sum_{j=1}^k\chi_j(-\log|s_{lc}|^2) + \epsilon\chi_{k}(-\log|s_{lc}|^2)-\log|s_{lc}|^2-\log|s_{klt}|^2\\
        &\lesssim -\log|s_{lc}|^2-\log|s_{klt}|^2.
    \end{aligned}
    \end{equation}
    Here we used the fact that $\chi_j$ is less singular than $\chi_0$ for any $j>0$. Then, we have
    \begin{equation}\label{estimate_logg^n}
        (\log(1+g_{k,\epsilon}))^n\lesssim(-\log|s_{lc}|^2-\log|s_{klt}|^2)^n\lesssim (-\log|s_{lc}|^2)^n \cdot (-\log|s_{klt}|^2)^n,
    \end{equation}
    and
    \begin{equation}\label{estimate_loglogg}
    \begin{aligned}
        \log(1+\log(1+g_{k,\epsilon}))&\lesssim\log(1-\log|s_{lc}|^2-\log|s_{klt}|^2)\\
        &= \chi_1(-\log|s_{lc}|^2-\log|s_{klt}|^2)-1.
    \end{aligned}
    \end{equation}
    We leave it for readers to check the following basic properties of $\chi_j$:
    \begin{itemize}
        \item $\chi_j\circ(\chi_k-1)=\chi_{j+k},\quad\forall j,k\geq0;$
        \item $\chi_j(a+b)\leq\chi_j(a)+b,\quad\forall a,b>1.$
    \end{itemize}
    Recall that $h_k(t):=\chi_0\cdot\chi_1\cdots\chi_{k-2}\cdot\chi_{k-2}^{\eta}(t)$, and we assumed that $|s_{klt}|^2,|s_{lc}|^2<e^{-1}$, then we have
    \begin{equation}\label{estimate_hloglogg}
        \begin{aligned}
            h_k(\log(1+\log(1+g_{k,\epsilon})))&\lesssim \prod_{j=0}^{k-2}\chi_j\circ[\chi_1(-\log|s_{lc}|^2-\log|s_{klt}|^2)-1] \cdot \chi_{k-2}^{\eta}\circ [\chi_1(-\log|s_{lc}|^2-\log|s_{klt}|^2)-1]\\
            &=\prod_{j=0}^{k-2} \chi_{j+1}(-\log|s_{lc}|^2-\log|s_{klt}|^2)\cdot \chi_{k-1}^{\eta}(-\log|s_{lc}|^2-\log|s_{klt}|^2)  \\
            &\leq\prod_{j=1}^{k-1} \left(\chi_{j}(-\log|s_{lc}|^2)-\log|s_{klt}|^2\right)\cdot \left(\chi_{k-1}(-\log|s_{lc}|^2)-\log|s_{klt}|^2  \right)^{\eta}\\
            &\lesssim\prod_{j=1}^{k-1}\chi_j(-\log|s_{lc}|^2)\cdot\chi_{k-1}^{\eta}(-\log|s_{lc}|^2) \cdot (-\log|s_{klt}|^2)^{k+\eta}.
        \end{aligned}
    \end{equation}
    Recall that our goal is to bound the integral $\int_Yw_k(g_{k,\epsilon})\ dV_Y$. Then, by estimates (\ref{estimate_g}), (\ref{estimate_logg^n}) and (\ref{estimate_hloglogg}), we are led to bound the following integral:
    \begin{equation*}\label{integral_to_bound}
        \begin{aligned}
            &\int_Y\frac{(-\log|s_{lc}|^2)^n \cdot (-\log|s_{klt}|^2)^n \cdot \left( \prod_{j=1}^{k-1}\chi_j(-\log|s_{lc}|^2) \right)^n \cdot \left( \chi_{k-1}(-\log|s_{lc}|^2) \right)^{n\eta} \cdot (-\log|s_{klt}|^2)^{n(k+\eta)}}{\left( \prod_{j=1}^k \chi_{j-1}(-\log|s_{lc}|^2) \right)^{n+\nu} \cdot \left( \chi_{k-1}(-\log|s_{lc}|^2) \right)^\epsilon \cdot |s_{lc}|^2 \cdot|s_{klt}|^2}dV_Y\\
            \lesssim&\int_Y\frac{1}{\left( \prod_{j=0}^{k-1} \chi_{j}(-\log|s_{lc}|^2) \right)^{\nu} \cdot \left( \chi_{k-1}(-\log|s_{lc}|^2) \right)^{\epsilon-n\eta} \cdot |s_{lc}|^2 } \cdot \frac{\left( -\log|s_{klt}|^2 \right)^{n(k+1+\eta)}}{|s_{klt}|^2}dV_Y.
        \end{aligned}
    \end{equation*}
    It suffices to check it locally. Let $U\subset\bC^n$ be the unit polydisc equipped with a coordinate system $\{z_1,z_2,\dots,z_n\}$, such that the divisors $E$ with $a_E=-1$ restricted to $U$ are defined by $(z_j=0)$, $j=1,\dots,p$; the divisors $E$ with $a_E>-1$ are defined by $(z_{p+j}=0)$, $j=1,\dots,s$. By the definition of $\nu$, we have $p\leq\nu$. We have $|s_{lc}|^2\sim\prod_{j=1}^p|z_j|^2$ and $|s_{klt}|^2\sim \prod_{j=1}^s|z_{p+j}|^{2b_j}$, where $b_j<1$ is $-1$ times the discrepancy of the divisor $(z_{p+j}=0)$. Then by Fubini's Theorem, the integral to bound becomes:
    $$\int_{U\cap\bC^p}\underbrace{\frac{d\lambda_{\bC^p}}{ \prod_{i=0}^{k-1} \left(\chi_{i}(\sum_{j=1}^p\t_j) \right)^{\nu} \cdot \left( \chi_{k-1}(\sum_{j=1}^p\t_j) \right)^{\epsilon-n\eta} \cdot \prod_{j=1}^p|z_j|^2 }}_{(A)} \cdot \int_{U\cap\bC^s}\underbrace{\frac{\left(\sum_{j=1}^s\t_{p+j}\right)^{n(k+1+\eta)}}{\prod_{j=1}^s|z_{p+j}|^{2b_j}}d\lambda_{\bC^s}}_{(B)}\cdot \int_{U\cap\bC^{n-p-s}}d\lambda_{\bC^{n-p-s}},$$
    where $\t_j=-\log|z_j|^2$. Then it suffices to show that term $(A)$ and term $(B)$ are integrable.

    For term $(A)$, we first claim that for any $i\geq 0$, we have 
    \begin{equation}\label{AMGM for chi}
        \chi_i(\sum_{j=1}^p\t_j)\gtrsim \prod_{j=1}^p\chi_i^{1/p}(\t_j).
    \end{equation}
    Indeed, when $i=0$, this follows from the AM-GM inequality. Now we assume (\ref{AMGM for chi}) is true for $\chi_{i-1}$, we have 
    \begin{align*}
        \chi_i(\sum_{j=1}^p\t_j)&=1+\log(\chi_{i-1}(\sum_{j=1}^p\t_j))\\
        &\gtrsim1+\log(\prod_{j=1}^p\chi_{i-1}^{1/p}(\t_j))\\
        &=\frac{1}{p}\sum_{j=1}^p \left(1+\log(\chi_{i-1}(\t_j))\right)\\
        &\geq \prod_{j=1}^p\chi_i^{1/p}(\t_j).
    \end{align*}
    Therefore, we conclude by induction. Then we have
    \begin{align*}
        (A)&\lesssim\int_{U\cap\bC^p}\frac{d\lambda_{\bC^p}}{ \prod_{i=0}^{k-1} \left(\prod_{j=1}^p\chi_{i}^{1/p}(\t_j) \right)^{\nu} \cdot \left( \prod_{j=1}^p\chi_{k-1}^{1/p}(\t_j) \right)^{\epsilon-n\eta} \cdot \prod_{j=1}^p|z_j|^2 }\\
        \text{(since $p\leq\nu$)}&\lesssim\int_{U\cap\bC^p} \prod_{j=1}^p \frac{d\lambda_{\bC^p}}{\prod_{i=0}^{k-1}\chi_i(\t_j) \cdot \chi_{k-1}^{\epsilon'}(\t_j)\cdot|z_j|^2}\\
        &=\prod_{j=1}^p  \int_{U\cap\bC^1} \frac{d\lambda_{\bC^p}}{\prod_{i=0}^{k-1}\chi_i(\t_j) \cdot \chi_{k-1}^{\epsilon'}(\t_j)\cdot|z_j|^2},
    \end{align*}
    where $\epsilon'=(\epsilon-n\eta)/\nu>0$. Thus, we are reduced to bound the $1$ dimensional integral:
    $$\int_{|z|<1} \frac{\frac{\sqrt{-1}}{2\pi}dz\wedge d\b z}{\prod_{i=0}^{k-1}\chi_i(-\log|z|) \cdot \chi_{k-1}^{\epsilon'}(-\log|z|)\cdot|z|^2}\ ,$$
    using polar coordinates, it becomes
    \begin{align*}
        &\int_0^{2\pi}\int_0^1\frac{r/\pi}{\chi_{k-1}^{\epsilon'}(-\log r)\cdot \prod_{i=0}^{k-1}\chi_i(-\log r)\cdot r^2}dr d\theta\\
        (t_1:=-\log r)=&2\int_0^\infty\frac{1}{\chi_{k-1}^{\epsilon'}(t_1)\cdot \prod_{i=0}^{k-1}\chi_i(t_1)}dt_1\\
        (t_2:=\log(1+t_1))=&2\int_0^\infty\frac{1}{\chi_{k-2}^{\epsilon'}(t_2)\cdot \prod_{i=0}^{k-2}\chi_i(t_2)}dt_2\\
        \cdots\\
        =&2\int_0^\infty \frac{1}{\chi_0^{1+\epsilon'}(t_{k+1})}dt_{k+1}=2\int_0^\infty \frac{1}{(1+t_{k+1})^{1+\epsilon'}}dt_{k+1}<\infty.
    \end{align*}
    
    For term $(B)$, since $b_j<1$, we may choose a small number $0<\delta\ll1$ such that $b_j+\delta<1,\ j=1,\dots,s$. Then
    \begin{align*}
        \int_{U\cap\bC^s}\frac{\left(\sum_{j=1}^s\t_{p+j}\right)^{n(k+1+\eta)}}{\prod_{j=1}^s|z_{p+j}|^{2b_j}}d\lambda_{\bC^s} &\lesssim \int_{U\cap\bC^s}\prod_{j=1}^s\frac{\left(\t_{p+j}\right)^{n(k+1+\eta)}}{|z_{p+j}|^{2b_j}}d\lambda_{\bC^s}\\
        &=\prod_{j=1}^s\int_{|z_j|<1} \frac{ |z_j|^{2\delta} \cdot (-\log|z_j|^2)^{n(k+1+\eta)}}{|z_{p+j}|^{2(b_j+\delta)}}d\lambda_{\bC}\\
        &\lesssim\prod_{j=1}^s\int_{|z_j|<1} \frac{1}{|z_{p+j}|^{2(b_j+\delta)}}d\lambda_{\bC}\\
        &\sim\prod_{j=1}^s\int_0^1 \frac{1}{r^{2(b_j+\delta)}}rdr <\infty.
    \end{align*}
    Thus, we are done.
\end{proof}

\section{Proof of Theorem A and B.}
    
    \begin{theo}[\textit{ = Theorem \ref{thma}}]\label{lowerbdd}\label{lowerbdd=thma}
        Let $X$ be a stable variety of dimension $n$, and $\omega_X \in c_1(K_X)$ be a smooth K\"ahler form. 
    Let $\omega=\omega_X+dd^c\varphi_{KE}$ be the K\"ahler-Einstein metric on $X$.
    Then for any $k\geq 1$ and $\epsilon>0$, there exists a constant $C_{k,\epsilon}$ such that 
        $$\varphi_{KE} \geq -(n+\nu_X)\sum_{j=1}^k\log^{(j)}(-\log\r) -\epsilon\log^{(k)}(-\log\r) + C_{k,\epsilon}.$$
    \end{theo}
    \begin{proof}
        We fix a log resolution $\pi:X\to Y$ such that $\nu_X(\pi)=\nu_X$. We use the same notation as in Section \ref{density}. Since $\psi_j\sim-\log^{(j)}(-\log\r)$, it suffices to show 
        $$\vphi_{KE}\geq(n+\nu_X)\sum_{j=1}^k\psi_j+\epsilon\psi_k+C_{k,\epsilon}$$
        Let $\theta=\pi^*\omega_X$, then the K\"ahler-Einstein potential satisfies the following Monge-Amp\`ere equation on $Y$:
        \begin{equation}\label{MA}
            \left\la(\theta + dd^c \pi^*\varphi_{KE})^n
            \right\ra=e^{\pi^*\varphi_{KE}} \cdot \prod_i|s_i|_{h_i}^{2a_i} \cdot dV_Y,
        \end{equation}
        where $dV_Y$ is a smooth volume form on $Y$, $s_i$ is a section of $\cO_Y(E_i)$ cutting out $E_i$, and $|\cdot|_{h_i}$ is a Hermitian metric on $\cO_Y(E_i)$. 
        
        By Lemma \ref{finiteE of psi_k} below, we may assume $\psi_j\in\cE^1(Y,\frac{1}{3(k-1)(n+\nu)}\theta)$, $j=1,\dots,k-1$; and $\psi_{k}\in\cE^1(Y,\frac{1}{3(n+\nu+\epsilon)}\theta)$. Then, by convexity of $\cE$, we have 
        \begin{equation}\label{temp1}
            \psi_{k,\epsilon}:=(n+\nu_X)\sum_{j=1}^k\psi_j +\epsilon\psi_{k}\in \cE(Y,\frac{2}{3}\theta).
        \end{equation}
        Let $g_{k,\epsilon}:=\exp\left( (n+\nu)\sum_{j=1}^{k}\pi^*\psi_j+\epsilon\pi^*\psi_{k} \right)\prod_i|s_i|_{h_i}^{2a_i}$. By Proposition \ref{integrable of g_km} and Corollary \ref{solve KEMA}, there exists a $u_{k,\epsilon}\in\psh(Y,\frac{1}{3}\theta)\cap L^{\infty}(Y)$ such that 
        $$\left\la\left(\frac{1}{3}\theta+dd^cu_{k,\epsilon}\right)^n\right\ra=e^{u_{k,\epsilon}}g_{k,\epsilon}dV_Y.$$
        Then by (\ref{temp1}) we have
        \begin{equation}
            \begin{aligned}
                \left\la\left(\theta+dd^c\left(\psi_{k,\epsilon}+u_{k,\epsilon} \right)\right)^n\right\ra &\geq\left\la\left(\frac{1}{3}\theta+dd^cu_{k,\epsilon}\right)^n\right\ra\\
                &=e^{u_{k,\epsilon}}g_{k,\epsilon}dV_Y\\
                &=e^{\psi_{k,\epsilon}+u_{k,\epsilon}}\prod_i|s_i|^{2a_i}dV_Y,
            \end{aligned}
        \end{equation}
        Note that $u_{k,\epsilon}\in\cE(Y,\frac{1}{3}\theta)$, and hence $\psi_{k,\epsilon}+u_{k,\epsilon}\in\cE(Y,\theta)$. Then by \hyperref[comparison]{Comparison Principle}, we have
        \begin{align*}
            \int_{\{\pi^*\vphi_{KE}<\psi_{k,\epsilon}+u_{k,\epsilon}\}} e^{\psi_{k,\epsilon}+u_{k,\epsilon}}\prod_i|s_i|^{2a_i}dV_Y &\leq \int_{\{\pi^*\vphi_{KE}<\psi_{k,\epsilon}+u_{k,\epsilon}\}}\MA(\psi_{k,\epsilon}+u_{k,\epsilon})\\
            &\leq\int_{\{\pi^*\vphi_{KE}<\psi_{k,\epsilon}+u_{k,\epsilon}\}}\MA(\pi^*\vphi_{KE})\\
            &=\int_{\{\pi^*\vphi_{KE}<\psi_{k,\epsilon}+u_{k,\epsilon}\}}e^{\pi^*\vphi_{KE}}\prod_i|s_i|^{2a_i}dV_Y.
        \end{align*}
        This forces $\pi^*\vphi_{KE}\geq \psi_{k,\epsilon}+u_{k,\epsilon}$ a.e. with respect to the measure $\prod|s_i|^{2a_i}dV_Y$, then it follows that $\pi^*\vphi_{KE}\geq \psi_{k,\epsilon}+u_{k,\epsilon}$ a.e. and hence everywhere by the basic property of psh functions. Since $u_{k,\epsilon}$ is bounded, we are done.
        \end{proof}

\begin{lemm}\label{finiteE of psi_k}
    For any $1>c>0$ and $k\geq 1$, after rescaling the Hermitian metric $h$ on $L$ if necessary, $\pi^*\psi_k$ defined in (\ref{define chi_k}) is a $c\theta$-psh function. Moreover, $\pi^*\psi_k\in\cE^1(Y,\theta)$.
\end{lemm}
\begin{proof}
    We first show $\psi_{k}\in\psh(Y,c\theta)$. Recall that $\psi_k=-\chi_k(-\log\r)$, and direct computation gives
\begin{equation}
\begin{aligned}
    \frac{d}{dt}[-\chi_k(-t)]&=\prod_{i=0}^{k-1}\frac{1}{\chi_i(-t)}>0,\quad\forall\ t<0 ;\\
    \frac{d^2}{dt^2}[-\chi_k(-t)]&=\prod_{i=0}^{k-1}\frac{1}{\chi_i(-t)} \cdot \sum_{p=0}^{k-1}\prod_{q=0}^p\frac{1}{\chi_q(-t)}>0,\quad\forall\ t<0.
\end{aligned}
\end{equation}
We denote $A_k(t)=\frac{d}{dt}[-\chi_k(-t)]$ and $B_k(t)=\frac{d^2}{dt^2}[-\chi_k(-t)]$. For simplicity, we will omit $\pi^*$ and view $\psi_k,\ \r$ as functions on $Y$. Let $\r_\epsilon=\r+\epsilon$, $\psi_{k,\epsilon}=-\chi_k(-\log\r_\epsilon)$. Then $\{\psi_{k,\epsilon}\}$ is a sequence of smooth functions decreasing to $\psi_k$ pointwise. Then it suffices to show that $\psi_{k,\epsilon}$ are $c\theta$-psh. By the basic chain rule, we have
\begin{equation}\label{ddc psi_kep}
    \begin{aligned}
        dd^c\psi_{k,\epsilon} &= B_k(\log\r_\epsilon)\cdot\frac{d\r\wedge d^c\r}{\r_\epsilon^2}+A_k(\log\r_\epsilon)\cdot \left[ \frac{\sum_{i=1}^k\la D\sigma_i,D\sigma_i\ra}{\r_\epsilon}-\frac{d\r\wedge d^c\r}{\r^2_\epsilon} -\frac{\r}{\r_\epsilon}\pi^*\Theta_h(L)\right]\\
        &\geq B_k(\log\r_\epsilon)\cdot\frac{d\r\wedge d^c\r}{\r_\epsilon^2}+A_k(\log\r_\epsilon)\cdot \left[ \frac{d\r\wedge d^c\r}{\r\cdot\r_\epsilon}-\frac{d\r\wedge d^c\r}{\r^2_\epsilon} -\frac{\r}{\r_\epsilon}\pi^*\Theta_h(L)\right]\\
        & \geq -A_k(\log\r_\epsilon)\cdot \frac{\r}{\r_\epsilon}\cdot\pi^*\Theta_h(L).
    \end{aligned}
\end{equation}
Now we may rescale $h$ such that $A(\log\r_\epsilon)<c/C_L$ for $\epsilon<<1$, where $C_L$ is a constant such that $\Theta_h(L)\leq C_L\omega_X$. Then we have 
$$-A_k(\log\r_\epsilon)\cdot \frac{\r}{\r_\epsilon}\cdot\pi^*\Theta_h(L)\geq -\frac{c\r}{C_L\r_\epsilon}\cdot C_L\theta\geq-c\theta.$$
Thus, $\psi_{k,\epsilon}\in\psh(Y,c\theta)$, and hence $\psi_{k}\in\psh(Y,c\theta)$.
    
Now we show that  $\psi_{k}$ has finite energy. By Lemma \ref{cap decay}, we need to control the capacity decay of $\psi_k$. By \cite{hironaka}, we can further blow up and get a proper holomorphic map $\pi':Y'\to Y$ such that $\pi\circ\pi'$ is a log resolution of $(X,\cI)$, namely a log resolution of $X$ with the additional property that $(\pi\circ\pi')^{-1}\cI\cdot\cO_{Y'}=\cO_{Y'}(-\sum b_j E'_j)$, where $b_j$ is a positive integer attached to every exceptional divisor $E'_j$. Since capacity increases under the pullback of a holomorphic map, it suffices to control the capacity decay on $Y'$. It is known that, up to a universal constant, the global capacity can be controlled by the local capacity in the sense of Bedford-Taylor \cite[Proposition 9.8]{GZdegenerateCMA}; then we are reduced to the case in the unit polydisc in $\bC^n$. On such a polydisc, we have $E'_j=(z_j=0)$, and $(\pi\circ\pi')^*\sigma_i$ trivialized to a holomorphic function $f_i$. Since $(\pi\circ\pi')^{-1}\cI\cdot\cO_{Y'}=\cO_{Y'}(-\sum b_j E'_j)$, we see that 
$$(\pi\circ\pi')^*\r=\sum_i |f_i|^2\cdot (\pi\circ\pi')^*h=\prod_j |z_j|^{2b_j}\cdot\sum_i|\tilde{f_i}|^{2}\cdot (\pi\circ\pi')^*h,$$ 
where $\sum_i|\tilde{f_i}|^{2}$ is a strictly positive smooth function.
Then $(\pi\circ\pi')^*\log\r\sim\sum_jb_j\log|z_j|^2$, and the result follows from \cite[Proposition 2.3]{HenriKE_Poincare+Cone_sing}.
\end{proof}

Now we turn to the proof of Theorem \ref{thmb}:

    \begin{theo}[\textit{ = Theorem \ref{thmb}}]\label{thmb=upperbdd}
        Let $X$ be as above. Let $d=\dim X_{sing}$.
    Assume $X$ admits a log resolution $\pi:Y\to X$ satisfying
    \begin{itemize}
        \item[(1)] every exceptional divisor of $\pi$ has discrepancy -1;
        \item[(2)] $\pi$ is moreover a log resolution of $(X,\cI)$, where $\cI$ is the ideal sheaf of $\lcs(X)$.
    \end{itemize}
    Then there is a constant $C>0$ such that
    $$\vphi_{KE}\leq -(n-d+1)\log(-\log\r)+C.$$ 
    \end{theo}
    \begin{proof}
        By assumption we have $K_Y=\pi^*K_X-\sum_iE_i$, where $E_i$ is either an exceptional divisor or sets above the codimension 1 singularities. The Monge-Amp\`ere equation (\ref{MA}) becomes
        $$\left\la(\theta + dd^c \pi^*\varphi_{KE})^n
            \right\ra=e^{\pi^*\varphi_{KE}}  \cdot \frac{dV_Y}{\prod_i|s_i|_{h_i}^{2}}\ .$$

        Let $\t=\log\pi^*\r$, $\psi=-\log(-\log\t)$, we are going to show that there exists a constant $C>0$ such that:
        \begin{equation}\label{MA of psi_1}
            \left\la (\theta+(n-d+1)dd^c\psi)^n \right\ra \leq \frac{e^{(n-d+1)\psi}CdV_Y}{\prod_i|s_i|_{h_i}^2}=\frac{CdV_Y}{(-\t)^{n-d+1} \prod_i|s_i|_{h_i}^2}\ .
        \end{equation}
        Once this is established, we can use \hyperref[comparison]{Comparison Principle} as before and deduce following estimate:
        $$\vphi_{KE}\leq (n-d+1)\psi+\log C.$$
        
        We first note that $\psi$ is smooth, and hence locally bounded, outside the analytic subset $A:=\cup_i\mathrm{supp}(E_i)$, then the left hand side of (\ref{MA of psi_1}) is just the trivial extension to $Y$ of the smooth form $(\theta|_{Y\setminus A}+(n-d+1)dd^c\psi|_{Y\setminus A})^n$, which is defined on $Y\setminus A$. Thus it suffices to show (\ref{MA of psi_1}) on $Y\setminus A$. 
        
        Let $\omega_Y$ be a K\"ahler form on $Y$ such that $\omega_Y\geq\theta$, we define $F\in\cC^\infty(Y\setminus A)$ by
        $$F:=\frac{(\theta+(n-d+1)dd^c\psi)^n}{\omega_Y^n}\cdot(-\t)^{n-d+1}\cdot\prod_{i}|s_i|^2_{h_i},$$
        then to show (\ref{MA of psi_1}) is equivalent to showing that $F$ is bounded from above. Since $F$ is only singular near $A$, it suffices to show it locally around $A$.
        
        Let $y\in A$, and let $U$ be a chart centered at $y$ equipped with a coordinate system $\{z_1,\dots,z_n\}$ such that $E_i$'s are defined by $(z_i=0)$, $i=1,\dots,p$. Then as in the proof of Lemma \ref{finiteE of psi_k}, we have $$\t=\sum_{i=1}^pb_i\log|z_i|^2+\log\sum_j|\tilde{f_j}|^2+\log\pi^*h,$$
        where $b_i$'s are positive integers for every $i=1,\dots,p$; and $\sum_j|\tilde{f_j}|^2$ is a non-vanishing smooth function. On $U\setminus A$ we have
        \begin{equation}\label{t and ddct}
        \begin{aligned}
             (-\t)^{n-d+1}\cdot\prod_i|s_i|^2_{h_i}&\lesssim\left(\sum_{i=1}^p(-\log|z_i|)\right)^{n-d+1}\cdot\prod_{i=1}^p|z_i|^2\leq C_1,\\
                dd^c\t&=\underbrace{dd^c\log\sum_j|\tilde{f_j}|^2}_{\text{smooth across $A$}}-\underbrace{\pi^*\Theta_h(L)}_{\geq0} \leq C_1\omega_Y,
        \end{aligned}
        \end{equation}
        for some $C_1>0$, and
        \begin{equation}\label{dt wedge dct}
        \begin{aligned}
            d\t\wedge d^c\t&=\sqrt{-1}\left( \sum_ib_i\frac{dz_i}{z_i}+\text{smooth term} \right)\wedge\left( \sum_ib_i\frac{d\b z_i}{\b z_i}+\text{smooth term} \right)\\
            \text{(by Cauchy-Schwarz)}&\leq C_2'\cdot\sum_{i=1}^p\frac{\sqrt{-1}dz_i\wedge d\b z_i}{|z_i|^2}+\text{smooth term}\\
            &\leq C_2\cdot\left(\sum_{i=1}^p\frac{\sqrt{-1}dz_i\wedge d\b z_i}{|z_i|^2}+\omega_Y \right),
        \end{aligned}
        \end{equation}
        for some $C_2>0$. 
        
        Now recall that
        $dd^c\psi=\frac{dd^c\t}{-\t}+\frac{d\t\wedge d^c\t}{(-\t)^2}$, and note that $d\t\wedge d^c\t$ has rank 1, then we have (on $X\setminus A$)
        \begin{equation}
            \begin{aligned}
                F&=\left[\left(\theta+\frac{dd^c\t}{-\t}\right)^n + n\left(\theta+\frac{dd^c\t}{-\t}\right)^{n-1}\wedge\frac{d\t\wedge d^c\t}{(-\t)^2}\right]\cdot\frac{(-\t)^{n-d+1}\prod_{i}|s_i|^2_{h_i}}{\omega_Y^n}\\
                &=
                \underbrace{\frac{(\theta+\frac{dd^c\t}{-\t})^n}{\omega_Y^n}(-\t)^{n-d+1}\prod_{i}|s_i|^2_{h_i}}_{(\text{I})}  +  \underbrace{\frac{n\left(\theta+\frac{dd^c\t}{-\t}\right)^{n-1}\wedge d\t\wedge d^c\t}{\omega_Y^n} \cdot(-\t)^{n-d-1}\prod_{i}|s_i|^2_{h_i}}_{(\text{II})}
            \end{aligned}
        \end{equation}
        For term (I), by (\ref{t and ddct}), we have 
        $$(\text{I})\leq\frac{(\omega_Y+\frac{C_1}{-\t}\omega_Y)^n}{\omega_Y^n}C_1\leq(1+C_1)^nC_1.$$
        For term (II), we have (on $U\setminus A$)
        \begin{equation*}
        \begin{aligned}
            (\text{II})&\lesssim\frac{(\theta+\frac{\omega_Y}{-\t})^{n-1}\wedge d\t\wedge d^c\t}{\omega_Y^n}\cdot(-\t)^{n-d-1}\prod_{i}|s_i|^2_{h_i}\\
            &=\sum_{j=0}^{n-1}\binom{n-1}{j} \frac{\theta^j\wedge\omega_Y^{n-j-1}\wedge d\t\wedge d^c\t}{\omega_Y^n}\cdot(-\t)^{j-d}\prod_{i}|s_i|^2_{h_i}\\
            \text{by (\ref{dt wedge dct})}&\leq \sum_{j=0}^{n-1}\binom{n-1}{j} \left[\underbrace{\sum_{i=1}^p\frac{\theta^j\wedge \omega_Y^{n-j-1}\wedge \sqrt{-1}dz_i\wedge d\b z_i}{|z_i|^2\omega_Y^n}\cdot(-\t)^{j-d}\prod_{i}|s_i|^2_{h_i}}_{(\text{II}_j)} + \underbrace{\frac{\theta^j\wedge\omega_Y^{n-j}}{\omega_Y^n}\cdot(-\t)^{j-d}\prod_{i}|s_i|^2_{h_i}}_{\text{(III)}}\right]
        \end{aligned}
        \end{equation*}
        The term (III) is obviously bounded by (\ref{t and ddct}). For the term ($\text{II}_j$), if $j\leq d$, we have 
        \begin{align*}
            (\text{II}_j)&\lesssim\sum_{i=1}^p\frac{\theta^j\wedge \omega_Y^{n-j-1}\wedge \sqrt{-1}dz_i\wedge d\b z_i}{|z_i|^2\omega_Y^n}\cdot\frac{1}{(-\t)^{d-j}}
        \end{align*}
        which is obviously bounded. 
        
        Now assume $j>d$, we make the following 
        \paragraph*{Claim} There exists a $C_3>0$ such that for each $1\leq i\leq p$, we have
        $$\theta^j\wedge \sqrt{-1}dz_i\wedge d\b z_i\leq C_3|z_i|^2\omega_Y^{j+1}.$$
        Then we have
        $$(\text{II}_j)\leq\sum_{i=1}^pC_3\cdot(-\t)^{j-d}\prod_{i}|s_i|^2_{h_i}\leq pC_3C_1.$$
        Thus we are done.
    
        \paragraph*{Proof of Claim}
         We embed a small neighborhood $V$ of $\pi(y)\in X$ into $\bC^N$ for some $N>0$, shrinking $U$ if necessary, we assume $U\subset \pi^{-1}(V)$. Then we get a holomorphic map $\tilde{\pi}:U\to \bC^N$. Given an index subset $I=(i_1,i_2,\dots,i_j)\subset\{1,2,\dots,N\}$ with $|I|=j$, we define the associated projection map $\mathrm{pr}_I:\bC^N\to\bC^j$ by 
         $$\mathrm{pr}_I(w_1,\dots,w_N)=(w_{i_1},\dots.w_{i_j})$$
         Let $f_I:=\mathrm{pr}_I\circ\tilde{\pi}$, then we have following diagram:

$$\tikzset{every picture/.style={line width=0.75pt}} %set default line width to 0.75pt        
\begin{tikzpicture}[x=0.75pt,y=0.75pt,yscale=-1,xscale=1]
%uncomment if require: \path (0,788); %set diagram left start at 0, and has height of 788
\useasboundingbox (60,300) rectangle (380,500);

% Text Node
\draw (136.37,208.33) node [anchor=north west][inner sep=0.75pt]   [align=left] { };
% Text Node
\draw (125,326.75) node [anchor=north west][inner sep=0.75pt]    {$( z_{i} =0)$};
% Text Node
\draw (244,326.75) node [anchor=north west][inner sep=0.75pt]    {$U$};
% Text Node
\draw (136,398.75) node [anchor=north west][inner sep=0.75pt]    {$X_{sing}$};
% Text Node
\draw (245,399.75) node [anchor=north west][inner sep=0.75pt]    {$V$};
% Text Node
\draw (337,398.75) node [anchor=north west][inner sep=0.75pt]    {$\mathbb{C}_w^{N}$};
% Text Node
\draw (340,475.75) node [anchor=north west][inner sep=0.75pt]    {$\mathbb{C}_w^{j}$};
% Text Node
\draw (120,475.75) node [anchor=north west][inner sep=0.75pt]    {$\mathrm{pr}_{I}( X_{sing})$};
% Text Node
\draw (297,423.75) node [anchor=north west][inner sep=0.75pt]  [font=\small,color={rgb, 255:red, 66; green, 106; blue, 179 }  ,opacity=1 ]  {$\textcolor[rgb]{0.26,0.42,0.7}{f_I}$};
% Text Node
\draw (352,432.75) node [anchor=north west][inner sep=0.75pt]  [font=\small]  {$\mathrm{pr}_{I}$};
% Text Node
\draw (299,352.75) node [anchor=north west][inner sep=0.75pt]  [font=\small]  {$\tilde{\pi }$};
% Text Node
\draw (73,415.75) node [anchor=north west][inner sep=0.75pt]  [font=\small,color={rgb, 255:red, 66; green, 106; blue, 179 }  ,opacity=1 ]  {$f_I|_{( z_{s} =0)}$};
% Text Node
\draw (238,363.75) node [anchor=north west][inner sep=0.75pt]  [font=\small]  {$\pi $};
% Text Node
\draw (260,324.75) node [anchor=north west][inner sep=0.75pt]    {$\subset \mathbb{C}_z^{n}$};
% Connection
\draw    (193,334.95) -- (239,334.48) ;
\draw [shift={(241,334.46)}, rotate = 179.42] [color={rgb, 255:red, 0; green, 0; blue, 0 }  ][line width=0.75]    (4.37,-1.32) .. controls (2.78,-0.56) and (1.32,-0.12) .. (0,0) .. controls (1.32,0.12) and (2.78,0.56) .. (4.37,1.32)   ;
\draw [shift={(193,334.95)}, rotate = 359.42] [color={rgb, 255:red, 0; green, 0; blue, 0 }  ][line width=0.75]      (0,-4.47) .. controls (-1.23,-4.47) and (-2.24,-3.47) .. (-2.24,-2.24) .. controls (-2.24,-1) and (-1.23,0) .. (0,0) ;
% Connection
\draw    (193,484.35) -- (335,484.35) ;
\draw [shift={(337,484.35)}, rotate = 180] [color={rgb, 255:red, 0; green, 0; blue, 0 }  ][line width=0.75]    (10.93,-3.29) .. controls (6.95,-1.4) and (3.31,-0.3) .. (0,0) .. controls (3.31,0.3) and (6.95,1.4) .. (10.93,3.29)   ;
% Connection
\draw    (153.59,348.35) -- (153.9,392.35) ;
\draw [shift={(153.91,394.35)}, rotate = 269.6] [color={rgb, 255:red, 0; green, 0; blue, 0 }  ][line width=0.75]    (10.93,-3.29) .. controls (6.95,-1.4) and (3.31,-0.3) .. (0,0) .. controls (3.31,0.3) and (6.95,1.4) .. (10.93,3.29)   ;
% Connection
\draw    (154.17,420.35) -- (154.81,469.35) ;
\draw [shift={(154.83,471.35)}, rotate = 269.26] [color={rgb, 255:red, 0; green, 0; blue, 0 }  ][line width=0.75]    (10.93,-3.29) .. controls (6.95,-1.4) and (3.31,-0.3) .. (0,0) .. controls (3.31,0.3) and (6.95,1.4) .. (10.93,3.29)   ;
% Connection
\draw    (262,342.17) -- (332.4,394.61) ;
\draw [shift={(334,395.81)}, rotate = 216.68] [color={rgb, 255:red, 0; green, 0; blue, 0 }  ][line width=0.75]    (10.93,-3.29) .. controls (6.95,-1.4) and (3.31,-0.3) .. (0,0) .. controls (3.31,0.3) and (6.95,1.4) .. (10.93,3.29)   ;
% Connection
\draw    (251.58,346.35) -- (251.9,393.35) ;
\draw [shift={(251.92,395.35)}, rotate = 269.61] [color={rgb, 255:red, 0; green, 0; blue, 0 }  ][line width=0.75]    (10.93,-3.29) .. controls (6.95,-1.4) and (3.31,-0.3) .. (0,0) .. controls (3.31,0.3) and (6.95,1.4) .. (10.93,3.29)   ;
% Connection
\draw    (270,407.35) -- (332,407.35) ;
\draw [shift={(334,407.35)}, rotate = 180] [color={rgb, 255:red, 0; green, 0; blue, 0 }  ][line width=0.75]    (4.37,-1.32) .. controls (2.78,-0.56) and (1.32,-0.12) .. (0,0) .. controls (1.32,0.12) and (2.78,0.56) .. (4.37,1.32)   ;
\draw [shift={(270,407.35)}, rotate = 0] [color={rgb, 255:red, 0; green, 0; blue, 0 }  ][line width=0.75]      (0,-4.47) .. controls (-1.23,-4.47) and (-2.24,-3.47) .. (-2.24,-2.24) .. controls (-2.24,-1) and (-1.23,0) .. (0,0) ;
% Connection
\draw    (183,407.35) -- (240,407.35) ;
\draw [shift={(242,407.35)}, rotate = 180] [color={rgb, 255:red, 0; green, 0; blue, 0 }  ][line width=0.75]    (4.37,-1.32) .. controls (2.78,-0.56) and (1.32,-0.12) .. (0,0) .. controls (1.32,0.12) and (2.78,0.56) .. (4.37,1.32)   ;
\draw [shift={(183,407.35)}, rotate = 0] [color={rgb, 255:red, 0; green, 0; blue, 0 }  ][line width=0.75]      (0,-4.47) .. controls (-1.23,-4.47) and (-2.24,-3.47) .. (-2.24,-2.24) .. controls (-2.24,-1) and (-1.23,0) .. (0,0) ;
% Connection
\draw    (349.58,420.35) -- (349.9,469.35) ;
\draw [shift={(349.92,471.35)}, rotate = 269.63] [color={rgb, 255:red, 0; green, 0; blue, 0 }  ][line width=0.75]    (10.93,-3.29) .. controls (6.95,-1.4) and (3.31,-0.3) .. (0,0) .. controls (3.31,0.3) and (6.95,1.4) .. (10.93,3.29)   ;
% Connection
\draw [color={rgb, 255:red, 66; green, 106; blue, 179 }  ,draw opacity=1 ]   (259.38,346.35) -- (297.95,405.09)(301.79,410.94) -- (340.37,469.68) ;
\draw [shift={(341.46,471.35)}, rotate = 236.71] [color={rgb, 255:red, 66; green, 106; blue, 179 }  ,draw opacity=1 ][line width=0.75]    (10.93,-3.29) .. controls (6.95,-1.4) and (3.31,-0.3) .. (0,0) .. controls (3.31,0.3) and (6.95,1.4) .. (10.93,3.29)   ;
% Connection
\draw [color={rgb, 255:red, 66; green, 106; blue, 179 }  ,draw opacity=1 ]   (142.55,348.35) .. controls (105.89,386.93) and (106.16,427.5) .. (143.37,470.06) ;
\draw [shift={(144.51,471.35)}, rotate = 228.3] [color={rgb, 255:red, 66; green, 106; blue, 179 }  ,draw opacity=1 ][line width=0.75]    (10.93,-3.29) .. controls (6.95,-1.4) and (3.31,-0.3) .. (0,0) .. controls (3.31,0.3) and (6.95,1.4) .. (10.93,3.29)   ;

\end{tikzpicture}$$
        By definition, $\omega_X$ extends to a K\"ahler form on $\bC^N$, then $\theta$ is the pull back of a K\"ahler form on $\bC^N$. Let $\omega_N=\sum_{i=1}^N \sqrt{-1}dw_i\wedge d\b w_i$ be the euclidean metric on $\bC^N$ and $\omega_j$ be the euclidean metric on $\bC^j$, then there exists a constant $C_4$ such that $\theta\leq C_4\tilde{\pi}^*\omega_N$. 
        By definition of $\mathrm{pr}_I$, we have $\omega_N^j=\sum_{|I|=j}\iota_I\mathrm{pr_I}^*(\omega_j)^j$, where $\iota_I$ is a constant. Then we have
        \begin{equation}\label{bbb}
            \theta^j\leq C_4^j\tilde{\pi}^*\omega_N^j = C_4^j\sum_{|I|=j}\iota_If_I^*(\omega_j)^j.
        \end{equation}
        Thus, it suffices to show 
        $$f_I^*(\omega_j)^j\wedge \sqrt{-1}dz_i\wedge d\b z_i\leq C_3|z_i|^2\omega_n^{j+1}.$$
        for some constant $C_3>0$. Let $f=(f_1,\dots,f_j):\mathbb C_z^n\to \mathbb C_w^j$, then
        $f_I^*\omega_j=\sqrt{-1}\sum_{\alpha=1}^j df_\alpha\wedge d\bar f_\alpha$, and we have
        \begin{equation}\label{aaa}
        \begin{aligned}
            (f_I^*\omega_j)^j&=j!(\sqrt{-1})^j df_1\wedge d\bar f_1\wedge\cdots\wedge df_j\wedge d\bar f_j\\
            &=j!(\sqrt{-1})^j(-1)^{\frac{j(j-1)}2}df_1\wedge\cdots\wedge df_j\wedge d\bar f_1\wedge\cdots\wedge d\bar f_j\\
            &=j!(\sqrt{-1})^j(-1)^{\frac{j(j-1)}2}\left(\sum_{|K|=j}J_K(f)\,dz_K\right)\wedge \left(\sum_{|L|=j}\overline{J_L(f)}\,d\b z_L\right)\\
            &\leq C_5\sum_{|K|=j}|J_K(f)|^2\,\sqrt{-1}dz_K\wedge d\bar z_K .
        \end{aligned}
        \end{equation}
        where
        $J_K(f)=\det\left(\frac{\partial f_\alpha}{\partial z_k}\right)_{\alpha=1,\dots,j;\ k\in K}$. Now we consider the map $f_I|_{(z_i=0)}:(z_i=0)\to\mathrm{pr}_I(X_{sing})$, since $\mathrm{pr}_I(X_{sing})$ has dimension $\leq d$, and we assumed $j>d$, then we have $(\omega_j)^j=0$ on the regular part of $\mathrm{pr}_I(X_{sing})$. Then it follows from \cite[Lemma 1.3]{Dem85} that $(f_I^*\omega_j|_{(z_i=0)})^j=0$. Then by (\ref{aaa}), we have $J_K(f)\overline{J_L(f)}=0$ on $(z_i=0)$ if $i\notin K,L$. In particular we have $|J_K(f)|^2=0$ on $(z_i=0)$ if $i\notin K$. Since $J_K(f)$ is a holomorphic function, we have
        \begin{equation}\label{ccc}
            |J_K(f)|^2\leq|z_i|^2C_6,\ \mathrm{if}\ i\notin K,
        \end{equation}
        for some constant $C_6>0$. Therefore, by (\ref{bbb}), (\ref{aaa}) and (\ref{ccc}), we have
        \begin{equation*}
            \begin{aligned}
                \theta^j\wedge \sqrt{-1}dz_i\wedge d\b z_i &\leq C_4^j\sum_{|I|=j}\iota_I C_5\sum_{|K|=j} \left|J_K(f)\right|^2 dz_K\wedge d\b z_K\wedge dz_i\wedge d\b z_i\\
                &=C_4^jC_5\sum_{\substack{|I|=|K|=j\\ i\notin K}}\iota_I \left|J_K(f)\right|^2 dz_K\wedge d\b z_K\wedge dz_i\wedge d\b z_i\\
                &\leq C_4^jC_5C_6\sum_{\substack{|I|=|K|=j\\ i\notin K}}\iota_I \left|z_i\right|^2 dz_K\wedge d\b z_K\wedge dz_i\wedge d\b z_i\\
                &\leq C_7|z_i|^2\omega_Y^{j+1},            
                \end{aligned}
        \end{equation*}
        for some $C_7>0$. Thus we are done.
    \end{proof}

\section{Proof of Theorem \ref{thmc}}
\begin{comment}
\begin{defi}
    Let $X$ be a compact complex manifold of dimension $n$, let $E=\sum_{i\in I}E_i$ be a reduced divisor with simple normal crossing support. Let $k\in\{1,2,\dots,n\}$ be a positive integer We say that $E$ admits a \textit{good covering of rank $k$} if there exists an open covering $\{U_\a\}_{\a\in \cA}$ of $E$ such that 
    \begin{itemize}
        \item[(a)] each $U_\a$ admits a coordinate system $\{z_1,\dots,z_n\}$ such that $E\cap U_\a$ is defined by $\{\prod_{i=1}^{p_\a}z_i=0\}$;
        \item[(b)] $p_\a=k,\quad\forall\a\in A$.
    \end{itemize}
\end{defi}
\end{comment}

In this section, we prove the following theorem:
\begin{theo}[\textit{ = Theorem \ref{thmc}}]\label{2sideest=thmc}
    Let $X$ be as in Theorem \ref{thma}. Let $d=\dim X_{sing}$. Assume $X$ admits a log resolution $\pi:Y\to X$ that satisfies the following two conditions:
    \begin{itemize}
        \item[(1)] every exceptional divisor of $\pi$ has discrepancy -1;
        \item[(2')] $\pi$ is a composition of finitely many blow-ups along smooth centers.
    \end{itemize}
    Let $K_Y=\pi^*K_X-\sum_{i\in I}E_i$, where $E_i$ is either an exceptional divisor or a divisor setting above the codimension 1 singularities. For every $i\in I$, let $s_i$ be a section of $\cO_Y(E_i)$ that cuts out $E_i$. Let $h_i$ be a Hermitian metric on $\cO_Y(E_i)$. Then we have
    $$-(n-d+1)\log(-\log\prod_{i\in I}|s_i|^2_{h_i})+\cO(1)\geq\pi^*\vphi_{KE}\geq-(n+\nu_X(\pi))\log(-\log\prod_{i\in I}|s_i|^2_{h_i})+\cO(1).$$
\end{theo}

We first fix some notation. By assumption, $X$ admits a log resolution $\pi:Y\to X$ which is a composition of finitely many blow-ups along smooth centers and such that every exceptional divisor has discrepancy -1. Let $\nu=\nu_X(\pi)$ (see Definition \ref{define_nu_X}). Let $\{E_i,\ i\in I\}$ denote the set of exceptional divisors. For every $i\in I$, let $s_i$ be a section of $\cO_Y(E_i)$ that cuts out $E_i$. Let $h_i$ be a Hermitian metric on $\cO_Y(E_i)$ which will be fixed later. Let $\omega_X\in c_1(K_X)$ be a K\"ahler form on $X$, $\theta:=\pi^*\omega_X$. Then the K\"ahler-Einstein potential satisfies the following Monge-Amp\`ere equation on $Y$:
\begin{equation}\label{isoMA}
    \la(\theta+\pi^*\vphi_{KE})^n\ra=\frac{e^{\pi^*\varphi_{KE}} dV_Y}{\prod_{i\in I}|s_i|_{h_i}^2}.
\end{equation}
Our proof is divided into 2 steps:
\begin{itemize}
    \item Construct a suitable $\theta$-psh function $\phi$ of finite energy on $Y$;
    \item Show that $\phi$ is a sub/super-solution to (\ref{isoMA}) (up to some constant);
\end{itemize}

\subsection*{Step 1: Construct a suitable $\theta$-psh function $\phi$ of finite energy on $Y$.}\label{step1} 
The following lemma will be useful:
\begin{lemm}
    Fix $m\leq|I|$. Then for every $i\in I$, there exists a vector $\bc_i=(c_{i}^1, c_{i}^2,\dots,c_{i}^m)\in(\bR_{>0})^m$ and a Hermitian metric $h_i$ on $\cO_Y(E_i)$ such that the following two conditions are satisfied:
    \begin{itemize}
        \item[1.]  $\forall K\subset I$, if $|K|\leq m$, then the matrix $[c_{k}^\ell]_{k\in K,\ 1\leq \ell\leq m}$ has rank $|K|$;
        \item[2.] $\forall \ell\in\{1,2,\dots,m\}$, $\frac{1}{m}\theta-\sum_{i\in I} c_{i}^\ell\Theta_{h_i}(E_i)$ is a K\"ahler form on $Y$.
    \end{itemize}
\end{lemm}
\begin{proof}
    Since $\pi$ is a composition of finitely many blow-ups along a smooth center, it is well known that there exists a vector $\ba=(a_i)_{i\in I}\in(\bR_{>0})^{|I|}$ such that $\frac{1}{m}\pi^*K_X-\sum_{i\in I}a_iE_i$ is ample. Then we can fix a Hermitian metric $h_i$ on $\cO_Y(E_i)$ such that $\frac{1}{m}\theta-\sum_{i\in I}a_i\Theta_{h_i}(E_i)>0$. Since K\"ahlerness is an open condition, we see that there is a neighbourhood $U\subset(\bR_{>0})^{|I|}$ of $\ba$ such that $\forall\bb=(b_i)_{i\in I}\in U$, $\frac{1}{m}\theta-\sum_{i\in I}b_i\Theta_{h_i}(E_i)$ is K\"ahler. Now we need to choose a point  $(\bb^1,\bb^2,\dots,\bb^m)\in U^m$ such that the first condition is satisfied.

    Given $K\subset I$, we have 
    $$\mathrm{rank}[b_{k}^\ell]_{k\in K,\ 1\leq \ell\leq m} <|K|\iff \mathrm{all\ } |K|\times|K|\mathrm{\ minors\ vanish}.$$
    The later condition defines a proper Zariski closed subset of $U^\nu$. Since there are only finitely many $K\subset I$, we see that the set 
    $$\left\{(\bb^1,\bb^2,\dots,\bb^m)\in U^m\ :\ \exists K\subset I \text{ such that } E_K\neq\varnothing\ \text{and }\mathrm{rank}[b_{k}^l]<|K| \right\}$$
    is also a proper Zariski closed set in $U^\nu$. Thus we can always choose a point $(\bc^1,\dots,\bc^m)\in U^m$ as desired.
\end{proof}

Now we let $m=\nu$ and choose $(c_i^l)_{i\in I,\ 1\leq l\leq \nu}$ as above. We define:
\begin{align*}
    \phi_\ell&=-\log(-\log\prod_{i\in I}|s_i|_{h_i}^{2c_i^\ell}),\quad 1\leq \ell\leq\nu;\\
    \phi&=\sum_{1\leq \ell\leq\nu} (1+\frac{n}{\nu})\phi_\ell.
\end{align*}

We note that $\phi_\ell\sim-\log(-\log\prod_{i\in I}|s_i|_{h_i}^{2})$ for every $\ell$. Then $\phi\sim-(n+\nu)\log(-\log\prod_{i\in I}|s_i|_{h_i}^{2})$, and hence it suffices to show $\phi$ is a lower bound. We first show that $\phi$ is $\theta$-psh. Consider $\phi_{\ell,\epsilon}:=-\log(-\log\prod_{i\in I}(|s_i|_{h_i}^{2}+\epsilon)^{c_i^\ell})$. Computing similarly as in (\ref{ddc psi_kep}), we have
\begin{align*}
    dd^c\phi_{\ell,\epsilon}&=\frac{d\log\prod_{i\in I}(|s_i|_{h_i}^{2}+\epsilon)^{c_i^\ell}\wedge d^c\log\prod_{i\in I}(|s_i|_{h_i}^{2}+\epsilon)^{c_i^\ell}}{\left(\log\prod_{i\in I}(|s_i|_{h_i}^{2}+\epsilon)^{c_i^\ell}\right)^2}+\frac{dd^c\log\prod_{i\in I}(|s_i|_{h_i}^{2}+\epsilon)^{c_i^\ell}}{-\log\prod_{i\in I}(|s_i|_{h_i}^{2}+\epsilon)^{c_i^\ell}}\\
    &\geq \frac{1}{-\log\prod_{i\in I}(|s_i|_{h_i}^{2}+\epsilon)^{c_i^\ell}} \cdot \sum_{i\in I}c_i^\ell dd^c\log(|s_i|_{h_i}^{2}+\epsilon)\\
    &=\frac{1}{-\log\prod_{i\in I}(|s_i|_{h_i}^{2}+\epsilon)^{c_i^\ell}}\cdot\sum_{i\in I} c_i^\ell \left( \frac{\epsilon\la Ds,Ds \ra}{(|s_i|^2_{h_i}+\epsilon)^2}-\frac{|s_i|^2_{h_i}}{|s_i|^2_{h_i}+\epsilon}\Theta_{h_i}(E_i) \right)\\
    &\geq\frac{1}{-\log\prod_{i\in I}(|s_i|_{h_i}^{2}+\epsilon)^{c_i^\ell}}\cdot\sum_{i\in I}\frac{|s_i|^2_{h_i}}{|s_i|^2_{h_i}+\epsilon}\cdot(-c_i^\ell\Theta_{h_i}(E_i)).
\end{align*}
Let $\epsilon\to0$, we get 
$$dd^c\phi_\ell\geq\frac{1}{-\log\prod_{i\in I}|s_i|_{h_i}^{2c_i^\ell}}\cdot\sum_{i\in I}-c_i^\ell\Theta_{h_i}(E_i).$$
Then we have
$$\frac{1}{n+\nu}\theta+dd^c\phi_\ell=\left(\frac{1}{n+\nu}-\frac{1}{-\log\prod_{i\in I}|s_i|_{h_i}^{2c_i^\ell}}\right)\theta +\frac{1}{-\log\prod_{i\in I}|s_i|_{h_i}^{2c_i^\ell}}\cdot(\theta-\sum_{i\in I}c_i^\ell\Theta_{h_i}(E_i))\geq 0$$
once we rescale $h_i$ such that $-\log\prod_{i\in I}|s_i|_{h_i}^{2c_i^\ell}\geq n+\nu$. Thus $\phi_\ell$ is $\frac{1}{n+\nu}\theta$-psh, and hence $\phi$ is $\theta$-psh.
Now the same proof of Lemma \ref{finiteE of psi_k} shows that $\phi\in\cE^1(Y,\theta)$.

\subsection*{Step 2: $\phi$ is a sub/super-solution}
We first establish an estimate of $\MA(\phi)$.

Let $A=\cup_i \mathrm{supp}(E_i)$. Since $\phi$ is locally bounded outside $A$, the non-pluripolar product $\la (\theta+dd^c\phi) \ra^n$ is the trivial extension of the measure $((\theta+dd^c\phi)|_{Y\setminus A})^n$. Thus, it suffices to do the computation on $Y\setminus A$, on which all things are smooth. 
Let $\t_\ell=\log\prod_{i\in I}|s_i|^{2c_i^l}_{h_i}$ and $\t=\log\prod_{i\in I}|s_i|^2_{h_i}$. Then it is easy to see that there is a $C_\t>0$ such that
\begin{equation}\label{Ctau}
    C_\t^{-1}\cdot \t\leq \t_\ell \leq C_\t\cdot\t,\quad\forall \ell\in\{1,\dots,\nu\}.
\end{equation}

On $Y\setminus A$, we have
\begin{equation}\label{computeMA(phi)}
\begin{aligned}
    (\theta+dd^c\phi)^n&=(\theta+\sum_{1\leq \ell\leq\nu} dd^c\phi_\ell)^n\\
    &=\left(\theta+\sum_{1\leq \ell\leq\nu}\frac{dd^c\t_\ell}{-\t_\ell} +\sum_{1\leq \ell\leq\nu}\frac{d\t_\ell\wedge d^c\t_\ell}{\t_\ell^2}\right)^n\\
    &=\left( (1-\sum_{1\leq \ell\leq\nu}\frac{1}{-\t_\ell})\theta+\sum_{1\leq \ell\leq\nu}\frac{\theta-\sum_ic_i^\ell\Theta_{h_i}(E_i)}{-\t_\ell} +\sum_{1\leq \ell\leq\nu}\frac{d\t_\ell\wedge d^c\t_\ell}{\t_\ell^2} \right)^n.
\end{aligned}
\end{equation}

Let $\omega_Y$ be a K\"ahler form on $Y$ such that $\omega_Y\geq\theta$. By Step 1, $\theta-\sum_{i\in I}c_i^\ell\Theta_{h_i}(E_i)$ is a K\"ahler form for every $\ell$, then we can find a constant $C_0>0$ such that 
\begin{equation}\label{C_0}
    C_0^{-1}\cdot \omega_Y\leq\theta-\sum_{i\in I}c_i^\ell\Theta_{h_i}(E_i)\leq C_0\cdot \omega_Y,\quad \forall \ell\in\{1,\dots,\nu\}.
\end{equation}

After rescaling $h_i$ if necessary, we may assume $\sum_\ell\frac{1}{-\t_\ell}\leq 1$. Then we have 
\begin{equation}\label{easy estimate}
    \theta\geq(1-\sum_{1\leq \ell\leq\nu}\frac{1}{-\t_\ell})\theta \geq0.
\end{equation}
Plug (\ref{C_0}), (\ref{Ctau}) and (\ref{easy estimate}) into (\ref{computeMA(phi)}) we get
\begin{equation}\label{2side est}
    \left( \theta+C_1\cdot\frac{\omega_Y}{-\t}+C_2\cdot\frac{1}{\t^2}\sum_{1\leq \ell\leq\nu}d\t_\ell\wedge d^c\t_\ell \right)^n
    \geq
    (\theta+dd^c\phi)^n
    \geq
    \left( C_1'\cdot\frac{\omega_Y}{-\t}+C_2'\cdot\frac{1}{\t^2}\sum_{1\leq \ell\leq\nu} d\t_\ell\wedge d^c\t_\ell \right)^n,
\end{equation}
where $C_1=\nu C_0C_\t^{-1}$, $C_2=\nu C_\t^{-2}$, $C_1'=\nu C_0^{-1}C_\t$, and $C_2'=\nu C_\t^2$. 

\paragraph*{Lower bound.} 
We show that $\phi+C'$ is a sub-solution for some constant $C'>0$, i.e.
$$(\theta+dd^c\phi)^n\geq\frac{e^{\phi+C'}dV_Y}{\prod_{i\in I}|s_i|_{h_i}^2}.$$
Since $\phi\in\cE^1(Y,\theta)$, by using \hyperref[comparison]{Comparison Principle} as before, this estimate implies $\pi^*\vphi_{KE}\geq\phi+C'$.
By (\ref{2side est}), it suffices to show
\begin{equation}\label{iso goal subsolution}
    \left( C_1'\cdot\frac{\omega_Y}{-\t}+C_2'\cdot\frac{1}{\t^2}\sum_{1\leq \ell\leq\nu} d\t_\ell\wedge d^c\t_\ell \right)^n\geq\frac{e^{\phi+C'}dV_Y}{\prod_{i\in I}|s_i|_{h_i}^2}.
\end{equation}
It suffices to show it locally near $A$. Let $y\in A$, Let $U$ be a chart centered at $y$ and equipped with a coordinate system $\{z_1,\dots,z_n\}$ such that the $E_i$'s that intersect with $U$ are defined by $(z_i=0),\ i=1,\dots,p$. Note that by definition we have $p\leq\nu$. On $U\setminus A$, we have
$$\frac{e^\phi dV_Y}{\prod_{i\in I}|s_i|^2_{h_i}}=\frac{dV_Y}{\prod_{1\leq \ell\leq v}\t_\ell^{1+\frac{n}{\nu}} \cdot\prod_{1\leq i\leq p}|z_i|^2\cdot(\mathrm{smooth\ term})}=\cO(1)\cdot\frac{\omega_{eucl}^n}{\t^{n+\nu}\cdot\prod_{1\leq i \leq p}|z_i|^2},$$
where $\omega_{eucl}=\sum_{1\leq i\leq n}\sqrt{-1}dz_i\wedge d\b{z}_i$ is the euclidean metric on $\bC^n$. Thus our goal (\ref{iso goal subsolution}) becomes: find constants $C'_U>0$ such that on $U\setminus A$ we have 
$$\left(\frac{\omega_Y}{-\t}+\frac{1}{\t^2}\sum_{1\leq \ell\leq\nu} d\t_\ell\wedge d^c\t_\ell \right)^n\geq\frac{C'_U\cdot\omega_{eucl}^n}{\t^{n+\nu}\cdot\prod_{1\leq i \leq p}|z_i|^2}.$$
On $U\setminus A$, we have
\begin{align*}
    \tau_\ell&=\log\left\{\prod_{1\leq i\leq p}|z_i|^{2c_i^\ell}\times\text{smooth function}\times\prod_{E_j\cap U=\varnothing}|s_j|^2_{h_j}\right\}\\
    &=\sum_{1\leq i\leq p}c_i^\ell\log|z_i|^2+\mathrm{smooth\ function},\\
    \d\tau_\ell&=\sum_{1\leq i\leq p}\frac{c_i^\ell dz_i}{z_i}+\gamma_\ell,
\end{align*}
where $\gamma_\ell$ is a smooth form across $A$. Let $B_U$ be a constant such that $\omega_Y\geq B_U\cdot\omega_{eucl}$ on $U$, then 
\begin{align*}
    \left(\frac{\omega_Y}{-\t}+\frac{1}{\t^2}\sum_{1\leq \ell\leq\nu} d\t_\ell\wedge d^c\t_\ell \right)^n&\geq\left(\frac{B_U\cdot\omega_{eucl}}{-\t}+\frac{1}{\t^2}\sum_{1\leq \ell\leq\nu} d\t_\ell\wedge d^c\t_\ell \right)^n\\
    &=\sum_{0\leq k\leq n}\binom{n}{k}\left( \frac{B_U\cdot\omega_{eucl}}{-\t} \right)^{n-k}\wedge\left( \frac{1}{\t^2}\sum_{1\leq \ell\leq\nu} d\t_\ell\wedge d^c\t_\ell \right)^k\\
    &\geq \left( \frac{B_U\cdot\omega_{eucl}}{-\t} \right)^{n-p}\wedge\left( \frac{1}{\t^2}\sum_{1\leq \ell\leq\nu} d\t_\ell\wedge d^c\t_\ell \right)^p\\
    &=\frac{B_U^{n-p}}{(-\t)^{n+p}}\omega_{eucl}^{n-p}\wedge\left(\sum_{1\leq \ell\leq\nu} d\t_\ell\wedge d^c\t_\ell\right)^p\\
    &\geq\frac{B_U^{n-p}}{(-\t)^{n+\nu}}\omega_{eucl}^{n-p}\wedge\left(\sum_{1\leq \ell\leq\nu} d\t_\ell\wedge d^c\t_\ell\right)^p.
\end{align*}
It follows from the definition of $[c_i^\ell]$ in \hyperref[step1]{Step 1} and Lemma \ref{linear algebra} below that 
$$\omega_{eucl}^{n-p}\wedge\left(\sum_{1\leq \ell\leq\nu} d\t_\ell\wedge d^c\t_\ell\right)^p\geq C'_3\frac{\omega_{eucl}^n}{\prod_{1\leq i\leq p}|z_i|^2}$$
Thus, by choosing $C_U'=B_U^{n-p}C_3'$, we are done.

\begin{lemm}\label{linear algebra}
    Let $\omega_{eucl}$ be the euclidean metric on $\bC^n$. Let $1\leq p\leq\nu\leq n$ be integers. Let $[c_i^\ell]_{1\leq i\leq p,\ 1\leq \ell\leq \nu}$ be a $p\times\nu$ matrix of rank $p$. Let $\alpha_\ell=\sum_{1\leq i\leq p} c_i^\ell\frac{ dz_i}{z_i}$ and let $\gamma_\ell$ be a smooth $(1,0)$-form on $\bC^n$, $\ell=1,\dots,\nu$. Then there is a neighbourhood $V$ of $0$ and a constant $C_V$ such that
    $$\omega_{eucl}^{n-p}\wedge\left(\sqrt{-1}\sum_{1\leq \ell\leq\nu} (\a_\ell+\gamma_l)\wedge (\bar\a_\ell+\bar\gamma_\ell)\right)^p\geq C_V\frac{\omega_{eucl}^n}{\prod_{1\leq i\leq p}|z_i|^2}$$
\end{lemm}
\begin{proof}
Let
$$
\beta_\ell:=\alpha_\ell+\gamma_\ell,
\qquad
\Theta:=\sqrt{-1}\sum_{\ell=1}^{\nu}\beta_\ell\wedge \overline{\beta_\ell}.
$$
We want to prove that
$$
\omega_{\mathrm{eucl}}^{n-p}\wedge \Theta^p
\geq
C_V\frac{\omega_{\mathrm{eucl}}^n}{\prod_{i=1}^p |z_i|^2}
$$
in a neighbourhood of $0$.

Since the $p\times \nu$ matrix $[c_i^\ell]$ has rank $p$, there exists an indice subset 
$$L_0=\{\ell_1,\dots,\ell_p:1\leq \ell_1<\cdots<\ell_p\leq \nu\}$$
such that
$$\det(c_i^{\ell_j})_{1\leq i,j\leq p}\neq 0.$$

We write locally
$\gamma_\ell=\sum_{j=1}^n g_j^\ell\,dz_j,$ 
where the coefficients $g_j^\ell$ are smooth functions. Then
$$
\beta_\ell=
\sum_{i=1}^p \left(\frac{c_i^\ell}{z_i}+g_i^\ell\right)dz_i + \sum_{j=p+1}^n g_j^\ell\,dz_j.
$$
Consider the $p\times p$ matrix 
$$A=\left[\frac{c_i^{\ell_j}}{z_i}+g_i^\ell\right]_{1\leq i,j\leq p},$$
its determinant is
$$
\det(A)(z)
:=
\det\left(
\frac{c_i^{\ell_j}}{z_i}+g_i^{\ell_j}(z)
\right)_{1\leq i,j\leq p}=
\frac{1}{z_1\cdots z_p}
\det\left(
c_i^{\ell_j}+z_i g_i^{\ell_j}(z)
\right)_{1\leq i,j\leq p}.
$$
Since
$
\det(c_i^{\ell_j})\neq 0,
$
the continuity of determinant implies that, after shrinking $V$ if necessary, we have
$$
\left|
\det\left(
c_i^{\ell_j}+z_i g_i^{\ell_j}(z)
\right)
\right|
\geq c_0
$$
on $V$, for some constant $c_0>0$. Hence
$$
|\det(A)|^2
\geq
\frac{c_0^2}{\prod_{i=1}^p |z_i|^2}.
$$
Now we have
\begin{align*}
    \omega_{eucl}^{n-p}\wedge \Theta^p&=\omega_{eucl}^{n-p}\wedge\left( \sqrt{-1}\sum_{\ell=1}^{\nu}\beta_\ell\wedge \overline{\beta_\ell} \right)^p\\
    &\geq \omega_{eucl}^{n-p}\wedge\left(\iota_{L_0}|\det(A)|^2 \bigwedge_{i=1}^p(\sqrt{-1}dz_i\wedge d\b z_i )\right)\\
    &\geq \iota_{L_0}(n-p)!\cdot|\det(A)|^2\,\omega_{eucl}^n,
\end{align*}
where $\iota_{L_0}$ is a constant depending on $L_0$. some constant \(C>0\). Combining the previous inequalities, we have
\[
\omega_{eucl}^{n-p}\wedge \Theta^p
\geq
C\,\frac{\omega_{eucl}^n}
{\prod_{i=1}^p |z_i|^2},
\]
where $C=\iota_{L_0}\cdot (n-p)!\cdot c_0^2$.
\end{proof}

\paragraph*{Upper bound.} 
Recall that for a fixed $1\leq \ell\leq\nu$, $\phi_\ell$ is $\frac{1}{n+\nu}\theta$-psh, then we see that $(n-d+1)\phi_\ell$ is $\frac{n-d+1}{n+\nu}\theta$-psh, and hence $\theta$-psh. It obviously has finite energy. By replacing $\psi$ with $\phi_\ell$ in the proof of Theorem \ref{thmb=upperbdd}, we see that $(n-d+1)\phi_\ell+C$ is a super-solution of $(\ref{isoMA})$ for some $C>0$. Hence we have $\pi^*\vphi_{KE}\leq(n-d+1)\phi_\ell+C$.

\subsection*{End of proof.} We have shown that 
$$(n-d+1)\phi_\ell+C\geq\pi^*\vphi_{KE}\geq\phi+C',$$
for some $C,C'>0$. Since $\phi_\ell=-\log(-\log\prod_{i\in I}|s_i|^2_{h_i})+\cO(1)$, we see that
$$-(n-d+1)\log(-\log\prod_{i\in I}|s_i|^2_{h_i})+\cO(1)\geq\pi^*\vphi_{KE}\geq-(n+\nu)\log(-\log\prod_{i\in I}|s_i|^2_{h_i})+\cO(1).$$
Thus, we are done.

    \bibliographystyle{smfalpha}
    \bibliography{biblio1}

\end{document}